\documentclass[10pt,reqno]{amsart}
\usepackage{amsmath}
\usepackage{amssymb}
\usepackage{amsthm}
\usepackage{eepic,epic}
\usepackage{graphics,psfrag,graphicx,color,subfigure,comment,enumerate,tcolorbox}
\usepackage{eepic,epic}
\newtheorem{thm}{Theorem}[section]
\newtheorem{cor}[thm]{Corollary}
\newtheorem{lem}[thm]{Lemma}
\newtheorem{prop}[thm]{Proposition}

\newtheorem{theoremalpha}{Theorem}

\theoremstyle{definition}

\theoremstyle{remark}
\newtheorem{rem}[thm]{Remark}
\numberwithin{equation}{section}

\begin{document}
\title[Splitting method of nonlinear heat equation]
{Convergence analysis of the splitting method to the nonlinear heat equation}

\author[Choi]{Hyung Jun Choi}
\address{School of Liberal Arts, Korea University of Technology and Education, Cheonan 31253, Republic of Korea}
\email{hjchoi@koreatech.ac.kr; choihjs@gmail.com}

\author[Choi]{Woocheol Choi}
\address{Department of Mathematics, Sungkyunkwan University, Suwon 16419, Republic of Korea}
\email{choiwc@skku.edu}

\author[Koh]{Youngwoo Koh}
\address{Department of Mathematics Education, Kongju National University, Kongju 32588, Republic
of Korea}
\email{ywkoh@kongju.ac.kr}

%\thanks{}
\maketitle

\begin{abstract}
In this paper, we analyze an operator splitting scheme of the nonlinear heat equation in $\Omega\subset\mathbb{R}^d$ ($d\geq 1$):
    $$
    \left\{\begin{aligned}
    &~\partial_t u = \Delta u + \lambda |u|^{p-1} u\quad\mbox{ in }\Omega\times(0,\infty),\\
    &~u=0\quad\mbox{ in }\partial\Omega\times(0,\infty),\qquad u ({\bf x},0) =\phi ({\bf x})\quad\mbox{ in }\Omega,
    \end{aligned}\right.
    $$
where $\lambda\in\{-1,1\}$ and $\phi \in W^{1,q}(\Omega)\cap L^{\infty} (\Omega)$ with $2\leq p < \infty$ and $d(p-1)/2<q<\infty$.
We establish the well-posedness of the approximation of $u$ in $L^r$-space ($r\geq q$), and furthermore, we derive its convergence rate of order $\mathcal{O}(\tau)$ for a time step $\tau>0$.
Finally, we give some numerical examples to confirm the reliability of the analyzed result 
\end{abstract}

%%%%%%%%%%%%%%%%%%%%%%%%%%%%%%%%%%%%%%%%%%%%%%%%%%%%%%%%%%%%%%%%%%%%%%%%%%%%%%%%%%%%%%%%%%%%%

\section{Introduction}\label{sec1}

Let $\Omega$ be any domain in $\mathbb{R}^{d}$, $d\geq1$.
Our concerned nonlinear heat equation is
    \begin{equation}\label{eq-main}
    \left\{\begin{aligned}
    &~\partial_t u = \Delta u + \lambda|u|^{p-1} u&&\mbox{ in }\Omega\times (0,\infty), \\
    &~u =0&&\mbox{ on }\partial\Omega\times (0,\infty), \\
    &~u({\bf x},0)= \phi ({\bf x})&&\mbox{ for }{\bf x} \in \Omega,
    \end{aligned}\right.
    \end{equation}
where $1<p<\infty$, $\lambda\in\{-1,1\}$ and $\phi({\bf x})$ is an initial temperature function.
So far, various mathematical issues of the nonlinear heat equation have been investigated.
Weissler \cite{W,WS} studied the local existence and uniqueness of solutions to the integral formula, where the initial datum $\phi$ is assumed to belong to $L^q(\Omega)$ for some $1\leq q<\infty$, and furthermore, the author \cite{WE} conducted research of the existence and non-existence of global solutions. Giga \cite{GS} showed a unique local regular solution in $L^r(0,T;L^s(\Omega))$, where $r$ and $s$ have the relation $1/r=(1/q-1/s)\,d/2$, $s>q$, provided that the initial function is in $L^q(\Omega)$ with $q=d(p-1)/2>1$, where $d$ is the space dimension. Brezis and Cazenave \cite{BC} discussed the local existence and uniqueness of solutions on a maximal interval $[0,T_{\rm max})$ with the $L^q(\Omega)$ ($1\leq q<\infty$) initial data, provided either $q>d(p-1)/2$ or $q=d(p-1)/2>1$. Gruji\'{c} and Kukavica \cite{GKS} considered the space-anlyticity radius of solutions in a smooth bounded domain on the initial condition of $L^q(\Omega)$, $1\leq q<\infty$. Ni and Sacks \cite{NSS} established the nonuniqueness and nonregularizing effect in some critical cases. We also refer to several literatures \cite{HY,CMR,GW,MM,MS} which are for the global existence, well-posedness and blow-up solutions.

We state the well-posedness of the solution $u$ to the problem \eqref{eq-main}, which is shown in \cite{BC, W} (also refer to \cite[Theorem 15.2]{QS}).

\begin{theoremalpha}\label{QS-thm}
Let $p>1$, $q\geq1$, $\frac{d(p-1)}{2} < q <\infty$ and $r \in [q, \infty]$.
If we assume that $\phi\in L^q(\Omega)$, then there exists a time $T_0>0$ such that the problem \eqref{eq-main} has a unique classical $L^r$-solution in $[0,T_0)$ and the following estimate holds:
    \begin{equation}\label{NLH_basic}
    \sup_{t\in[0,T_0)} t^{\frac{d}{2}\left(\frac{1}{q}-\frac{1}{r}\right)} \| u(t) \|_{L^{r}(\Omega)}
    \leq C_{d,p,q} \|\phi\|_{L^q(\Omega)},
    \end{equation}
where $C_{d,p,q}>0$ is a constant independent of the domain $\Omega$.
Furthermore, the time $T_0 =T_0\left(d,p,q,\|\phi\|_{L^q(\Omega)}\right)$ in \eqref{NLH_basic} can be precisely determined by
    \begin{equation}\label{def_T0}
    T_0
    = c_{d,p,q} \left( 1/~ \|\phi\|_{L^q(\Omega)}^{(p-1)q} \right)^{\frac{1}{q - \frac{d(p-1)}{2}} }
    \end{equation}
for some positive constant $c_{d,p,q}>0$.
\end{theoremalpha}

\vspace{0.3cm}

In this paper, we are concerned with the operator splitting scheme of the nonlinear heat equation \eqref{eq-main}, which is regarded as a fundamental problem consisting of the diffusion part and the nonlinear reaction part. Such numerical method is useful in the numerical computation of the semilinear-type problem, and it can be proposed by splitting \eqref{eq-main} into a linear flow and a nonlinear one as follows (cf. \cite{GC}):

\vspace{0.2cm}

\begin{tcolorbox}
\noindent{\bf (Linear Part)}
    Let $v_0({\bf x})\in\mathbb{R}$ be a given function.
    For $t>0$, the operator $S(t)$ is defined by $S(t)v_0 = e^{t\Delta}v_0$ which denotes the solution $v$ satisfying the following linear heat propagation:
    $$
    \left\{ \begin{aligned}  &\partial_t v  = \Delta v && \mbox{ in }\Omega\times(0,\infty),\\
    &v=0&&\mbox{ on }\partial\Omega\times(0,\infty),\\
    &v({\bf x},0)= v_0({\bf x})&&\mbox{ for }{\bf x}\in  \Omega.
    \end{aligned}
    \right.
    $$

\noindent{\bf (Nonlinear Part)}
    For a bounded function $w_0({\bf x})\in\mathbb{R}$, there exists a time $T_1>0$ and a unique solution $w$ satisfying
    \begin{equation}\label{NL-flow}
    \left\{ \begin{aligned}
    &\partial_t w = \lambda |w|^{p-1} w&&\mbox{ in }\Omega\times(0,T_1),\\
    &w=0&&\mbox{ on }\partial\Omega\times(0,T_1),\\
    &w({\bf x},0)= w_0({\bf x}) &&\mbox{ for }{\bf x}\in\Omega.
    \end{aligned}
    \right.
    \end{equation}
    The operator $N(t)$ is defined by $N(t)w_0=w$, where $w$ is the solution of \eqref{NL-flow}.
    Indeed, the explicit form of $w$ in \eqref{NL-flow} can be expressed by
    \begin{equation}\label{NL-flow_2}
    N(t)w_0
    = w_0({\bf x}) \left( \frac{1}{1-(p-1)\lambda t|w_0({\bf x})|^{p-1}} \right)^{\frac{1}{p-1}}\quad\mbox{ for }t\in(0,T_1).
    \end{equation}
    Here, the time $T_1 =T_1\left(p,\lambda,\|w_0\|_{L^\infty(\Omega)}\right)  >0$ in \eqref{NL-flow} can be determined by
    \begin{equation}\label{def_T1}
    T_1 = \left( (p-1)\left|\lambda\right| \|w_0\|_{L^{\infty}(\Omega)}^{p-1} \right)^{-1}.
    \end{equation}
\end{tcolorbox}

\vspace{0.3cm}

Using the operators $N(t)$ and $S(t)$, the solution $u$ of \eqref{eq-main} is separated with a small switching time $\tau\ll T_1$ as follows: On a fixed time interval $[0,T_1)$, we define the Lie approximation $Z(n\tau)$ for $n\in\mathbb{N}$ and $0<n\tau<T_1$ by
    \begin{equation}\label{operator-splitting-approximation}
    Z(n \tau) \phi = (S(\tau) N(\tau))^n \phi.
    \end{equation}
We note that $Z(n \tau) \phi$ is well-defined for $n\in\mathbb{N}$ satisfying $0<n\tau<T_1$ (see Proposition \ref{prop_def_N} later).
Furthermore, the Duhamel-type formula for $Z$ is given by
    \begin{equation}\label{eq-1-9}
    Z(n \tau) = S(n \tau) + \tau \sum_{k=0}^{n-1} S(n\tau -k \tau)\left( \frac{N(\tau) - I}{\tau}\right)
    Z(k \tau),
    \end{equation}
where $I$ denotes the identity operator.
On the other hand, the Duhamel formula of the solution $u$ to \eqref{eq-main} is obtained by
    \begin{equation}\label{eq-1-8}
    u(t) = S(t) \phi + \int_0^t S(t-s) \left( \lambda |u(s)|^{p-1} u (s) \right) ds,\qquad t \geq 0.
    \end{equation}

The aim of this paper is to establish a convergence analysis of the approximation $Z(n\tau)\phi$ defined by \eqref{operator-splitting-approximation}, which is depending on the regularity of the initial datum $\phi$.
Regarding the concerned problem \eqref{eq-main}, various numerical strategies and numerically analyzed results have been investigated in numerous literatures.
Mizuguchi et al. \cite{MTKO1,MTKO2} verified the existence and local uniqueness of mild solutions, and demonstrated a global-in-time solution for a certain semilinear parabolic equation.
Kyza and Metcalfe \cite{KMP} considered an adaptive space-time numerical approach, based on a rigorous a posteriori error bound, with a general local Lipschitz reaction term whose solution may blow up in finite time.
Furthermore, the splitting methods of other equations such as the reaction diffusion equation and the nonlinear Schr\"{o}dinger equation have been studied in numerous literatures \cite{BCCM, CCDV, L} and other references therein.
Recently, many authors analyzed convergence results of the Schr\"{o}dinger equation with the initial datum in low-regularity spaces (cf. \cite{BBD, CK, I2, IZ1, IZ2, I}).
In addition, regarding the nonlinear parabolic problem, some convergence results of the approximation based on the proposed splitting scheme can be found in the references \cite{BCCM, CCDV, D, DDDLLM, DR, RS}.
%, but we do not have rich knowledge of its convergence under the initial datum of low-regularity up to $W^{1,q}(\Omega)$, which is the object of this paper.

From now on, we give two main results, i.e., the well-posedness of $Z(n\tau)\phi$ in \eqref{operator-splitting-approximation} and its error estimate.
To compare with $u(n\tau)$ in Theorem \ref{QS-thm}, we need to consider the common time interval $(0,T_2)$, where $T_2=T_2\left(d,p,q,\lambda, \|\phi\|_{L^q(\Omega)}, \|\phi\|_{L^\infty(\Omega)}\right)>0$ is defined by
    \begin{equation}\label{def_T2}
    T_2 = \min \{T_0,\, T_1\}\in(0,\infty),
    \end{equation}
where $T_0$ and $T_1$ are given by \eqref{def_T0} and \eqref{def_T1}, respectively.
As shown in Appendix later, the well-posedness of the approximation $Z(n\tau)\phi$ on the interval $(0,T_2)$ is stated as follows:

\begin{thm}\label{main_thm_2}
Let $p>1$, $q\geq1$, $\frac{d(p-1)}{2}<q<\infty$ and $r\in[q,\infty]$.
If we assume that $\phi \in L^{q}(\Omega) \cap L^{\infty}(\Omega)$, then there exists a constant $C_{d,p,q}>0$, independent of the domain $\Omega$, such that
    \begin{equation}\label{loc_wp}
    \sup_{0< n\tau <T_2} (n\tau)^{\frac{d}{2} \left(\frac{1}{q}-\frac{1}{r}\right)} \|Z(n\tau)\phi\|_{L^r(\Omega)} \leq
    C_{d,p,q} \|\phi\|_{L^q(\Omega)},
    \end{equation}
where $\tau \in (0, T_2/2]$ is a given small switching time and $T_2>0$ is given by \eqref{def_T2}.
\end{thm}

Next, we show the convergence result depending on the regularity of the initial datum $\phi$.
We carefully try to measure the difference between the solution $u(n\tau)$ and the approximation $Z(n\tau)\phi$, based on the Duhamel formulas \eqref{eq-1-8} and \eqref{eq-1-9}. In addition, using the inequality \eqref{loc_wp}, we can derive the following convergence result of $Z(n\tau)\phi$ , which will be proved in Section \ref{sec3} later.

%The proof of this result is inspired by the reference \cite{I} which is for the splitting method of the nonlinear Schr\"{o}dinger equation (see also \cite{CK}).
%

\begin{thm}\label{main_thm}
Let $p\geq 2$, $q\geq1$, $\frac{d(p-1)}{2}<q<\infty$ and $r\in[q,\infty]$.
If we assume that $\phi \in W^{1,q}(\Omega) \cap L^{\infty}(\Omega)$, then there exists a constant $C_{d,p,q}>0$, independent of the small switching time $\tau<1$ and the domain $\Omega$, such that
    \begin{equation}\label{main_thm_est'}
    \begin{split}
    &\sup_{0<n\tau<T_2} (n\tau)^{\frac{d}{2} \left(\frac{1}{q}-\frac{1}{r}\right) } \| u(n\tau) - Z(n\tau)\phi \|_{L^r(\Omega)}
        \\
     &\qquad\qquad \leq C_{d,p,q} \, T_2^{1 -\frac{d(p-1)}{2q}}\tau
    \|\phi\|_{W^{1,q}(\Omega)}^p \left(1 + \| \phi\|_{W^{1,q}(\Omega)}^{p-1} \right),
    \end{split}
    \end{equation}
where $\tau \in (0, T_2/2]$ is a given small switching time and $T_2>0$ is given by \eqref{def_T2}.
\end{thm}

\begin{rem}
As seen in Section \ref{sec3} later, the following error estimate will be proved:
    \begin{align}
    &\sup_{0<n\tau<T_2} (n\tau)^{\frac{d}{2} \left(\frac{1}{q}-\frac{1}{r}\right) - \left(1-\mu\right) } \| u(n\tau) - Z(n\tau)\phi \|_{L^r(\Omega)} \label{main_thm_est}\\
    &\qquad\leq C_{d,p,q} \tau
    \left( \| \phi\|_{L^q(\Omega)}^{p-2} \| \phi\|_{W^{1,q}(\Omega)}^{p+1}
    + \| \phi\|_{L^q(\Omega)}^{p-2} \| \phi\|_{W^{1,q}(\Omega)}^{2}
    + \tau \|\phi \|_{L^{\infty}(\Omega)}^{p-1}   \| \phi\|_{L^q(\Omega)}^{p-2} \| \phi\|_{W^{1,q}(\Omega)}^{2} \right),\nonumber
    \end{align}
    where $\mu:=\frac{d(p-1)}{2q}<1$.
Since $\tau\leq n\tau < T_2\leq T_1\leq 1/\|\phi\|_{L^\infty(\Omega)}^{p-1}$ and $\| \phi\|_{L^q(\Omega)} \leq \| \phi\|_{W^{1,q}(\Omega)}$, the inequality \eqref{main_thm_est} obviously implies the main result \eqref{main_thm_est'}.
\end{rem}

\begin{rem}
When $\Omega$ is a bounded domain,
the assumption $\phi \in L^\infty(\Omega)$ in Theorem \ref{main_thm} is not necessary due to $W^{1,q}(\Omega)\hookrightarrow L^\infty(\Omega)$.
\end{rem}

\vspace{0.1cm}

This paper is organized as follows.
In Section \ref{sec2}, we give some properties regarding the operators $S(t)$ and $N(t)$, which are essentially used in the proof of Theorem \ref{main_thm}.
In Section \ref{sec3}, we prove the main result: Theorem \ref{main_thm} by the use of induction.
In Section \ref{sec4}, we try to confirm the analyzed result \eqref{main_thm_est'} by some numerical experiments.

%%%%%%%%%%%%%%%%%%%%%%%%%%%%%%%%%%%%%%%%%%%%%%%%%%%%%%%%%%%%%%%%%%%%%%%%%%%%%%%%%%%%%%%%%%%%%%%%%%%%%%%%%%%%%%%

\section{Preliminaries}\label{sec2}

In this section, we discuss some properties of the linear heat flow $S(t)$ and nonlinear flow $N(t)$.
From Proposition 48.4 in \cite{QS}, we first give the following basic property of the flow $S(t)$: Let $p\geq1$ and $r \in [q,\infty]$ with $q\geq 1$. If we assume that $\phi\in L^q(\Omega)$, then we have
    \begin{equation}\label{linear_heat}
    \| S(t)\phi \|_{L^{r}(\Omega)}
    \leq (4\pi t)^{-\frac{d}{2}\left(\frac{1}{q}-\frac{1}{r}\right)} \|\phi\|_{L^q(\Omega)}\qquad\mbox{ for }t>0.
    \end{equation}
By \eqref{linear_heat}, we show the useful estimate of $u$ to the problem \eqref{eq-main}, which is similar to \eqref{NLH_basic}.

\begin{prop}\label{regular_bdd_prop}
Let $p\geq 2$, $q\geq1$, $\frac{d(p-1)}{2} < q <\infty$ and $r \in [q, \infty]$.
If we assume that $\phi\in W^{1,q}(\Omega)$, then there exists a time $T_0>0$ as in \eqref{def_T0}, satisfying
    \begin{equation}\label{NLH_modi_1}
    \sup_{t\in[0,T_0)} t^{\frac{d}{2}\left(\frac{1}{q}-\frac{1}{r}\right)} \| u(t) \|_{W^{1,r}(\Omega)} \leq
    C_{d,p,q} \|\phi\|_{W^{1,q}(\Omega)},
    \end{equation}
where $C_{d,p,q}>0$ is a constant independent of the domain $\Omega$.
\end{prop}

\begin{proof}
We now derive the estimate \eqref{NLH_modi_1}.
Due to the result \eqref{NLH_basic}, it is enough to show
    \begin{equation}\label{prop_2_1_claim}
    \sup_{t\in[0,T_0)} t^{\frac{d}{2}\left(\frac{1}{q}-\frac{1}{r}\right)} \| \nabla u(t) \|_{L^r(\Omega)} \leq
    C_{d,p,q} \|\phi\|_{W^{1,q}(\Omega)}.
    \end{equation}
By a direct calculation, the equality \eqref{eq-1-8} implies
    $$
    \nabla u(t) = S(t)(\nabla\phi) + \int_0^t S(t-s) \left(\nabla \left(\lambda |u(s)|^{p-1} u (s) \right)\right) ds .
    $$
Using \eqref{linear_heat}, we obtain
    \begin{align}
    &t^{\frac{d}{2}\left(\frac{1}{q}-\frac{1}{r}\right)} \|\nabla u(t)\|_{L^r(\Omega)}\nonumber\\
    &\qquad\leq C_d \| \nabla \phi\|_{L^q(\Omega)} + t^{\frac{d}{2}\left(\frac{1}{q}-\frac{1}{r}\right)} \int_0^t \left\|
    S(t-s)\left( \nabla \left( \lambda |u(s)|^{p-1} u (s) \right) \right) \right\|_{L^r(\Omega)} ds\nonumber \\
    &\qquad\leq C_d \|\phi\|_{W^{1,q}(\Omega)} + C_{d,p,q} t^{\frac{d}{2}\left(\frac{1}{q}-\frac{1}{r}\right)} \int_0^t
    (t-s)^{-\frac{d}{2}\left(\frac{1}{q}-\frac{1}{r}\right)} \left\| \nabla\left( \lambda |u(s)|^{p-1} u (s) \right) \right\|_{L^q(\Omega)} ds .\label{prop_2_1_ineq1}
    \end{align}
Since
    $$
    \nabla\left(|u(s)|^{p-1} u (s)\right) = \left( (p-1)|u(s)|^{p-2}u(s) + |u(s)|^{p-1} \right) \nabla u(s),
    $$
and by H\"older's inequality and \eqref{NLH_basic} with $p\geq 2-\frac{q}{r}$ and $q\leq r$, one yields
    \begin{align}
    &\left\| \nabla\left(|u(s)|^{p-1} u (s)\right) \right\|_{L^q(\Omega)} \nonumber\\
    &\qquad\leq p \left\| |u(s)|^{p-1} |\nabla u(s)| \right\|_{L^q(\Omega)} \nonumber\\
    &\qquad\leq p \left( \|u(s)\|_{L^{\frac{rq(p-1)}{r-q}}(\Omega)}\right)^{p-1} \| \nabla u(s) \|_{L^r(\Omega)} \nonumber\\
    &\qquad\leq C_{d,p,q} \left( s^{-\frac{d}{2}\left(\frac{1}{q}-\frac{r-q}{rq(p-1)}\right)}
    \|\phi\|_{L^q(\Omega)}\right)^{p-1} \| \nabla u(s) \|_{L^r(\Omega)} \nonumber\\
    &\qquad\leq C_{d,p,q} \left( s^{-\frac{d}{2}\left(\frac{p-2}{q}+\frac{1}{r}\right)} \|\phi\|_{L^q(\Omega)}^{p-1} \right)
    s^{-\frac{d}{2}\left(\frac{1}{q}-\frac{1}{r}\right)} \left( \sup_{s\in[0,T_0)}
    s^{\frac{d}{2}\left(\frac{1}{q}-\frac{1}{r}\right)} \| \nabla u(s) \|_{L^r(\Omega)} \right).\label{prop_2_1_ineq2}
    \end{align}
On the other hand, we note that for $\alpha$, $\beta \in [0,1)$ and $t>0$,
    \begin{align}
    \int_0^t (t-s)^{-\alpha} s^{-\beta} ds
    &\leq \int_0^{t/2} \left(\frac{t}{2}\right)^{-\alpha} s^{-\beta} ds
    + \int_{t/2}^t (t-s)^{-\alpha} \left(\frac{t}{2}\right)^{-\beta} ds \nonumber\\
    &=C_{\alpha,\beta}~ t^{-\alpha-\beta+1},\label{basic_cal_1}
    \end{align}
where $C_{\alpha,\beta}:=2^{\alpha+\beta-1}\left(\frac{1}{1-\alpha}+\frac{1}{1-\beta}\right)>0$.
By \eqref{prop_2_1_ineq2} and \eqref{basic_cal_1}, and since $\frac{d}{2}\left(\frac{1}{q}-\frac{1}{r}\right) <1$ and $\frac{d(p-1)}{2q} <1$ for $p\geq2$ and $\frac{d(p-1)}{2} <q$, the inequality \eqref{prop_2_1_ineq1} becomes
    \begin{align}
    & t^{\frac{d}{2}(\frac{1}{q}-\frac{1}{r})} \|\nabla u(t)\|_{L^r(\Omega)} \nonumber\\
    &\qquad\leq C_d \|\phi\|_{W^{1,q}(\Omega)}
    + C_{d,p,q} \|\phi\|_{L^q(\Omega)}^{p-1} \left( \sup_{s\in[0,T_0)}
    s^{\frac{d}{2}\left(\frac{1}{q}-\frac{1}{r}\right)} \| \nabla u(s) \|_{L^r(\Omega)} \right)\nonumber\\
    &\qquad\qquad\times t^{\frac{d}{2}\left(\frac{1}{q}-\frac{1}{r}\right)}\int_0^t
    (t-s)^{-\frac{d}{2}\left(\frac{1}{q}-\frac{1}{r}\right)} s^{-\frac{d(p-1)}{2q}} ds \nonumber\\
    &\qquad\leq C_d \|\phi\|_{W^{1,q}(\Omega)}
    + C_{d,p,q} \|\phi\|_{L^q(\Omega)}^{p-1} \left( \sup_{s\in[0,T_0)}
    s^{\frac{d}{2}\left(\frac{1}{q}-\frac{1}{r}\right)} \| \nabla u(s) \|_{L^r(\Omega)} \right)
    t^{1-\frac{d(p-1)}{2}\frac{1}{q}}.\label{prop_2_1_ineq3}
    \end{align}
Using the definition \eqref{def_T0} of $T_0$, one sees that for $t<T_0$,
    \begin{align}
    C_{d,p,q} \|\phi\|_{L^q(\Omega)}^{p-1} t^{1-\frac{d(p-1)}{2}\frac{1}{q}}
    &\leq C_{d,p,q} \|\phi\|_{L^q(\Omega)}^{p-1} T_0^{1-\frac{d(p-1)}{2}\frac{1}{q}} \nonumber\\
    &\leq C_{d,p,q} \|\phi\|_{L^q(\Omega)}^{p-1} c_{d,p,q} \left( 1/\|\phi\|_{L^q(\Omega)}^{(p-1)q} \right)^{\frac{1}{q}}\nonumber\\
    &\leq \frac{1}{2},\label{cal_assum_T0}
    \end{align}
where $c_{d,p,q}>0$ is a sufficiently small constant given in Theorem \ref{QS-thm}.
From \eqref{cal_assum_T0}, the inequality \eqref{prop_2_1_ineq3} gives that for $t\in[0,T_0)$,
    \begin{equation}\label{prop_2_1_ineq4}
    t^{\frac{d}{2}\left(\frac{1}{q}-\frac{1}{r}\right)} \|\nabla u(t)\|_{L^r(\Omega)}
    \leq C_d\|\phi\|_{W^{1,q}(\Omega)}
    + \frac{1}{2} \left( \sup_{s\in[0,T_0)} s^{\frac{d}{2}\left(\frac{1}{q}-\frac{1}{r}\right)} \| \nabla
    u(s) \|_{L^r(\Omega)} \right),
    \end{equation}
and then taking the supremum of \eqref{prop_2_1_ineq4} on $[0,T_0)$, the desired estimate \eqref{prop_2_1_claim} follows.
\end{proof}

\begin{cor}\label{precise_original}
Let $p\geq 2$, $q\geq1$ and $\frac{d(p-1)}{2} < q <\infty$.
Suppose that $\phi\in W^{1,q}(\Omega)$.
Then for $t\in [0,T_0)$, we have the following estimates:
    $$
    \begin{aligned}
    (i)&~\left\| |u(t)|^{p-2} |\nabla u(t)|^{2} \right\|_{L^q(\Omega)}
    \leq C_{d,p,q} t^{-\mu} \| \phi\|_{L^q(\Omega)}^{p-2} \| \phi\|_{W^{1,q}(\Omega)}^{2} ,\\
    (ii)&~\left\| |u(t)|^{2p-1} \right\|_{L^q(\Omega)}
    \leq C_{d,p,q} t^{-\mu} \| \phi\|_{L^q(\Omega)}^{p-2} \| \phi\|_{W^{1,q}(\Omega)}^{p+1},
    \end{aligned}
    $$
where $\mu = \frac{d(p-1)}{2q} <1$ and $C_{d,p,q}>0$ is a constant not depending on the domain $\Omega$. Furthermore, if we assume that $\phi \in W^{1,q}(\Omega)\cap L^{\infty}(\Omega)$ and $0< t <T_2 \leq 1/\|\phi\|_{L^{\infty}(\Omega)}^{p-1}$, then we have
\begin{equation}
    (iii)~\left\| |u(t)|^{2p-3} |\nabla u(t)|^{2} \right\|_{L^q(\Omega)}
    \leq C_{d,p,q} t^{-\mu} \|\phi\|_{L^{\infty}(\Omega)}^{p-1}\|\phi\|_{L^q (\Omega)}^{p-2} \|\phi \|_{W^{1,q}(\Omega)}^2.
\end{equation}
\end{cor}

\begin{proof}
First, we use H\"older's inequality and apply the estimates \eqref{NLH_basic} and \eqref{NLH_modi_1} to deduce
\begin{equation}
\begin{split}
\left\| |u(t)|^{p-2}|\nabla u(t)|^2\right\|_{L^q (\Omega)} &\leq \|u(t)\|_{L^{\infty}(\Omega)}^{p-2} \|\nabla u(t) \|_{L^{\infty}(\Omega)} \|\nabla u(t)\|_{L^q (\Omega)}
\\
&\leq C_{d,p,q} t^{-\frac{d(p-1)}{2q}} \|\phi\|_{L^q (\Omega)}^{p-2} \|\phi \|_{W^{1,q}(\Omega)}^2,
\end{split}
\end{equation}
which is the estimate $(i)$.
Also, using H\"older's inequality, \eqref{NLH_basic} and Sobolev embedding \cite{A}, the second estimate $(ii)$ is obtained as follows:
    $$
    \begin{aligned}
    \left\| |u(t)|^{2p-1} \right\|_{L^q(\Omega)}
    &\leq \| u(t)\|_{L^{(2p-1)q}(\Omega)}^{p-2} \| u(t)\|_{L^{(2p-1)q}(\Omega)}^{p+1} \\
    &\leq C_{d,p,q} t^{-\frac{d(p-1)}{2q}} \| \phi\|_{L^q(\Omega)}^{p-2} \| \phi\|_{L^{\frac{(p+1)q}{2}}(\Omega)}^{p+1} \\
    &\leq C_{d,p,q} t^{-\frac{d(p-1)}{2q}} \| \phi\|_{L^q(\Omega)}^{p-2} \| \phi\|_{W^{1,q}(\Omega)}^{p+1},
    \end{aligned}
    $$
where the last inequality holds true since $d\left( \frac{1}{q} - \frac{2}{(p+1)q}\right) = \frac{d(p-1)}{(p+1)q} <1$ for $p \geq 2$ and $q> \frac{d(p-1)}{2}$.

To show the estimate $(iii)$, we recall that $\|S(t) v\|_{L^{\infty}(\Omega)} \leq \|v\|_{L^{\infty}(\Omega)}$ for any $v \in L^{\infty}(\Omega)$ (see e.g. Proposition 48.4 in \cite{QS}).
By this inequality, the solution $u(t)$ formulated by \eqref{eq-1-8} is estimated to be
$$
\|u(t)\|_{L^{\infty}(\Omega)} \leq \|\phi \|_{L^{\infty}(\Omega)} + \int_0^t \|u(s)\|_{L^{\infty}(\Omega)}^{p} ds, \qquad \forall~ t \geq 0.
$$
So, it follows by the standard argument that
\begin{equation}
\|u(t)\|_{L^{\infty}(\Omega)} \leq C \|\phi\|_{L^{\infty}(\Omega)}\qquad \textrm{for }\,0< t < T_2 \leq 1/\|\phi\|_{L^{\infty}(\Omega)}^{p-1}.
\end{equation}
Combining this with H\"older's inequality and the estimate $(i)$, one yields
\begin{equation}
\begin{split}
\left\| |u(t)|^{2p-3} |\nabla u(t)|^{2} \right\|_{L^q(\Omega)}
&\leq  \|u(t)\|_{L^{\infty}(\Omega)}^{p-1}\left\| |u(t)|^{p-2}|\nabla u(t)|^2\right\|_{L^q (\Omega)}
\\
&\leq C_{d,p,q} t^{-\frac{d(p-1)}{2q}} \|\phi\|_{L^{\infty}(\Omega)}^{p-1}\|\phi\|_{L^q (\Omega)}^{p-2} \|\phi \|_{W^{1,q}(\Omega)}^2,
\end{split}
\end{equation}
which gives the estimate $(iii)$.
\end{proof}

%%%%%%%%%%%%%%%%%%%%%%%%%%%%%%%%%%%%%%%%%%%%%%%%%%%%%%%%%%%%%%%%%%%%%%%%%%%%%%%%%%%%%%%%%%%%%%%%%%%%%%%%%%%%%%%

Next, we show the well-definedness of the operator $Z(k\tau)$ for $k=1,2,\cdots,N$.

\begin{prop}\label{prop_def_N}
Suppose that $\phi \in L^{\infty}(\Omega)$ and there is a number $N\in \mathbb{N}$ such that $N\tau<T_1$ for a given time step $\tau>0$.
Then $Z(k\tau) \phi = (S(\tau) N(\tau))^k \phi$ is well-defined for $k=1,2,\cdots, N$.
\end{prop}

\begin{proof}
First, we consider the case of $\lambda>0$.
To show this proposition, we shall use an induction argument.
Let $M>0$ be a constant given by $\|\phi\|_{L^\infty(\Omega)}^{p-1} \leq M$.

\noindent\underline{\it Step 1 (Base case)}.
From \eqref{linear_heat} and \eqref{NL-flow_2}, we have
    \begin{align}
    \left\| S(\tau)N(\tau)\phi \right\|_{L^\infty(\Omega)}^{p-1}
    &\leq \left\| N(\tau)\phi \right\|_{L^\infty(\Omega)}^{p-1}\nonumber\\
    &\leq \frac{\| \phi \|_{L^\infty(\Omega)}^{p-1}}{1-(p-1)\lambda\tau \| \phi \|_{L^\infty(\Omega)}^{p-1}}\nonumber\\
    &\leq \frac{M}{1-(p-1)\lambda\tau M}.\label{prop_2_3_ineq1}
    \end{align}

\noindent\underline{\it Step 2 (Inductive step)}.
Assume that for any $k=1,2,\cdots, N-1$,
    $$
    \left\| \left(S(\tau)N(\tau)\right)^k\phi \right\|_{L^\infty(\Omega)}^{p-1}
    \leq \frac{M}{1-(p-1)\lambda k\tau M} .
    $$
Then we obtain
    $$
    \begin{aligned}
    \left\| \left(S(\tau)N(\tau)\right)^{k+1}\phi \right\|_{L^\infty(\Omega)}^{p-1}
    &= \left\| S(\tau)N(\tau) \left(S(\tau)N(\tau)\right)^k\phi \right\|_{L^\infty(\Omega)}^{p-1} \\
    &\leq \frac{\left\| \left(S(\tau)N(\tau)\right)^k\phi \right\|_{L^\infty(\Omega)}^{p-1}}{1-(p-1)\lambda\tau \left\| \left(S(\tau)N(\tau)\right)^k\phi \right\|_{L^\infty(\Omega)}^{p-1}} \\
    &\leq \frac{M}{1-(p-1)\lambda (k+1)\tau M}.
    \end{aligned}
    $$
Since $N\tau <T_1 = \left( (p-1)\lambda M \right)^{-1}$ from \eqref{def_T1}, one gets
    \begin{equation}\label{prop_2_3_ineq2}
    \left\| \left(S(\tau)N(\tau)\right)^{N}\phi \right\|_{L^\infty(\Omega)}^{p-1}
    \leq \frac{M}{1-(p-1)\lambda (N\tau) M}
    < \infty.
    \end{equation}
Using the induction argument with \eqref{prop_2_3_ineq1} and \eqref{prop_2_3_ineq2}, $Z(k\tau)\phi$ is well-defined in the case of $\lambda>0$.
Furthermore, the case of $\lambda\leq 0$ is obvious, so this proposition is concluded.
\end{proof}

\begin{lem}\label{lem-2-3}
Let $p>1$ and $\lambda \in \{-1,1\}$.
Assume that
\begin{equation}\label{eq-2-21}
0 < \tau  \leq \frac{1}{2(p-1)} \min\left\{ \frac{1}{|u|^{p-1}},~\frac{1}{|v|^{p-1}}\right\}.
\end{equation}
Then, there exists a constant $c_p>0$ such that
    \begin{equation}\label{basic_ineq_1}
    \left| \left(\frac{N(\tau) - I}{\tau}\right) u - \left(\frac{N(\tau)-I}{\tau}\right) v \right|
    \leq c_p |u-v| \left( |u|^{p-1} + |v|^{p-1} \right),
    \end{equation}
and
    \begin{equation}\label{basic_ineq_2}
    \left| \left(\frac{N(\tau) - I}{\tau}\right) u - \lambda |u|^{p-1}u \right|
    \leq c_p \tau |u|^{2p-1} .
    \end{equation}
\end{lem}

\begin{proof}
From \eqref{NL-flow_2}, one sees that
    $$
    \begin{aligned}
    &\left(\frac{N(\tau) - I}{\tau}\right) u - \left(\frac{N(\tau)-I}{\tau}\right) v \\
    &\qquad= \frac{u-v}{\tau}\left( \left( \frac{1}{1-(p-1)\lambda\tau |u|^{p-1}} \right)^{\frac{1}{p-1}} -1 \right)\\
    &\qquad\qquad + \frac{v}{\tau} \left( \left( \frac{1}{1-(p-1)\lambda\tau |u|^{p-1}} \right)^{\frac{1}{p-1}} - \left(
        \frac{1}{1-(p-1)\lambda\tau |v|^{p-1}} \right)^{\frac{1}{p-1}} \right) .
    \end{aligned}
    $$
By \eqref{eq-2-21}, we notice that $(p-1)\tau |u|^{p-1}<1/2$ and $(p-1)\tau |v|^{p-1} <1/2$.
Without loss of generality, it is assumed that $0\leq |v(x)| \leq |u(x)|$.
Since
    $$
    a^{\frac{1}{p-1}} - b^{\frac{1}{p-1}} \leq 2p(a-b) \left( a^{\frac{1}{p-1}-1} +
    b^{\frac{1}{p-1}-1} \right)
    $$
for all $0\leq b \leq a$ and $p>1$, we have
    $$
    \begin{aligned}
    &\left| \left(\frac{N(\tau) - I}{\tau}\right) u - \left(\frac{N(\tau)-I}{\tau}\right) v \right| \\
    &\qquad\leq \frac{|u-v|}{\tau}\left| \frac{1}{1-(p-1)\lambda\tau |u|^{p-1}} -1 \right| \left( \left(
    \frac{1}{1-(p-1)\lambda\tau |u|^{p-1}} \right)^{\frac{-p+2}{p-1}} +1 \right) \\
    &\qquad\qquad + \frac{|v|}{\tau} \left| \frac{1}{1-(p-1)\lambda\tau |u|^{p-1}} - \frac{1}{1-(p-1)\lambda\tau |v|^{p-1}} \right|\\
    &\qquad\qquad \times\left( \left( \frac{1}{1-(p-1)\lambda\tau |u|^{p-1}} \right)^{\frac{-p+2}{p-1}} + \left(
    \frac{1}{1-(p-1)\lambda\tau |v|^{p-1}} \right)^{\frac{-p+2}{p-1}} \right),
    \end{aligned}
    $$
and then it implies
    \begin{equation}\label{lemma_2_4_ineq1}
    \left| \left(\frac{N(\tau) - I}{\tau}\right) u - \left(\frac{N(\tau)-I}{\tau}\right) v \right|\leq \frac{2|u-v|}{\tau} A_1 + \frac{2|v|}{\tau}A_2,
    \end{equation}
where
$$
\begin{aligned}
A_1&:=\left| \frac{1}{1-(p-1)\lambda\tau |u|^{p-1}} -1 \right|,\\
A_2&:=\left| \frac{1}{1-(p-1)\lambda\tau |u|^{p-1}} - \frac{1}{1-(p-1)\lambda\tau |v|^{p-1}} \right|.
\end{aligned}
$$
Indeed, a direct calculation gives that
    \begin{equation}\label{basic_diff_cal}
    A_1= \left| \frac{(p-1)\lambda\tau|u|^{p-1}}{1-(p-1)\lambda\tau |u|^{p-1}} \right|\leq 2(p-1) \tau|u|^{p-1}
    \end{equation}
and
    \begin{align}
    A_2&= \left| \frac{(p-1)\lambda\tau \left( |u|^{p-1} -|v|^{p-1}\right)} { \left(1-(p-1)\lambda\tau |u|^{p-1}\right) \left(1-(p-1)\lambda\tau |v|^{p-1}\right)} \right| \nonumber\\
    &\leq 4(p-1)\tau \left( |u| -|v|\right) \left( |u|^{p-2} +|v|^{p-2}\right) \nonumber\\
    &\leq 4(p-1)\tau |u -v| \left( |u|^{p-2} +|v|^{p-2}\right).\label{basic_diff_cal2}
    \end{align}
By \eqref{basic_diff_cal} and \eqref{basic_diff_cal2}, the inequality \eqref{lemma_2_4_ineq1} becomes
    $$
    \left| \left(\frac{N(\tau) - I}{\tau}\right) u - \left(\frac{N(\tau)-I}{\tau}\right) v \right|
    \leq 4(p-1) |u-v| \left( |u|^{p-1} + 2|u|^{p-2}|v| + 2|v|^{p-1} \right),
    $$
so the desired estimate \eqref{basic_ineq_1} follows.
Furthermore, one sees
    \begin{equation}\label{lemma_2_4_eq3}
    \left(\frac{N(\tau) - I}{\tau}\right) u - \lambda|u|^{p-1}u
    = \frac{u}{\tau}\left( \left( \frac{1}{1-(p-1)\lambda\tau |u|^{p-1}} \right)^{\frac{1}{p-1}} -1 \right) -
    \lambda |u|^{p-1}u,
    \end{equation}
and Taylor series expansion of $f(x):= \left( \frac{1}{1-(p-1)\lambda\tau x} \right)^{\frac{1}{p-1}}$ gives that for some $x_0\in(0,x)$,
    \begin{equation}\label{lemma_2_4_ineq2}
    \left( \frac{1}{1-(p-1)\lambda\tau x} \right)^{\frac{1}{p-1}}
    = 1 +\lambda\tau x + \frac{p}{2}\left( \frac{1}{1-(p-1)\lambda\tau x_0} \right)^{\frac{1}{p-1}+2} \lambda^2 \tau^2 x^2.
    \end{equation}
Here, the inequality \eqref{lemma_2_4_ineq2} can be derived by seeing
    $$
    \begin{aligned}
    f'(x)
    &= \lambda\tau \left( \frac{1}{1-(p-1)\lambda\tau x} \right)^{\frac{1}{p-1}+1} ,\\
    f''(x)
    &= p\lambda^2\tau^2 \left( \frac{1}{1-(p-1)\lambda\tau x} \right)^{\frac{1}{p-1}+2}.
    \end{aligned}
    $$
By \eqref{lemma_2_4_ineq2}, and since $(p-1)\tau x_0 \leq (p-1)\tau |u|^{p-1} \leq 1/2$, the quantity of \eqref{lemma_2_4_eq3} is estimated to be
    $$
    \begin{aligned}
    \left| \left(\frac{N(\tau) - I}{\tau}\right) u - \lambda|u|^{p-1}u \right|
    &= \left| \frac{u}{\tau}\left( f\left(|u|^{p-1}\right) -1 \right) -
    \lambda |u|^{p-1}u \right| \\
    &= \left| \frac{p}{2}\left( \frac{1}{1-(p-1)\lambda\tau x_0} \right)^{\frac{1}{p-1}+2} \lambda^2 \tau |u|^{2(p-1)}u \right| \\
    &\leq 4p \tau |u|^{2p-1},
    \end{aligned}
    $$
and then the inequality \eqref{basic_ineq_2} is obtained.
\end{proof}

\begin{lem}\label{lem-2-4}
Let $p\geq 2$ and $\lambda \in \{-1,1\}$.
For any $\tau\in(0,T_1 /2)$, there is a constant $c_p>0$ such that the solution $u$ of \eqref{eq-main} satisfies the following inequality:
    $$
    \left| (\partial_t - \Delta) \left( \left(\frac{N(\tau) - I}{\tau}\right) u(t) \right) \right|
    \leq c_p \left( |u|^{2p-1} + |u|^{p-2}|\nabla u|^2 + \tau |u|^{2p-3} |\nabla u|^2 \right).
    $$
\end{lem}

\begin{proof}
By a direct calculation, one sees
    $$
    \begin{aligned}
    \partial_t \left( \left(\frac{N(\tau) - I}{\tau}\right) u(t) \right)
    &= \frac{u_t}{\tau}\left( \left( \frac{1}{1-(p-1)\lambda\tau |u|^{p-1}} \right)^{\frac{1}{p-1}} -1 \right)\\
    &\qquad + \left( \frac{1}{1-(p-1)\lambda\tau |u|^{p-1}} \right)^{\frac{1}{p-1}+1} (p-1)\lambda |u|^{p-1} u_t,\\
    \nabla \left( \left(\frac{N(\tau) - I}{\tau}\right) u(t) \right)
    &= \frac{\nabla u}{\tau}\left( \left( \frac{1}{1-(p-1)\lambda\tau |u|^{p-1}} \right)^{\frac{1}{p-1}} -1
    \right)\\
    &\qquad + \left( \frac{1}{1-(p-1)\lambda\tau |u|^{p-1}} \right)^{\frac{1}{p-1}+1} (p-1)\lambda |u|^{p-1} \nabla u ,
    \end{aligned}
    $$
and
    \begin{equation}
    \begin{aligned}
    \Delta \left( \left(\frac{N(\tau) - I}{\tau}\right) u(t) \right)
    &= \frac{\Delta u}{\tau}\left( \left( \frac{1}{1-(p-1)\lambda\tau |u|^{p-1}} \right)^{\frac{1}{p-1}} -1
    \right)\\
    &\qquad + \left( \frac{1}{1-(p-1)\tau |u|^{p-1}} \right)^{\frac{1}{p-1}+1} (p-1)\lambda |u|^{p-2} \frac{u}{|u|} (\nabla u)^2 \\
    &\qquad + \left( \frac{1}{1-(p-1)\lambda \tau |u|^{p-1}} \right)^{\frac{1}{p-1}+2} p(p-1)^2 \lambda^2 \tau |u|^{2p-3} \frac{u}{|u|} (\nabla u)^2 \\
    &\qquad+\left( \frac{1}{1-(p-1)\lambda\tau |u|^{p-1}} \right)^{\frac{1}{p-1}+1} (p-1)^2 \lambda |u|^{p-2} \frac{u}{|u|} (\nabla u)^2\\
    &\qquad +\left( \frac{1}{1-(p-1)\tau |u|^{p-1}} \right)^{\frac{1}{p-1}+1} (p-1)\lambda |u|^{p-1} \Delta u.
    \end{aligned}
    \end{equation}
Since $(\partial_t - \Delta)u = \lambda |u|^{p-1}u$, we have
    $$
    \begin{aligned}
    (\partial_t - \Delta) \left( \left(\frac{N(\tau) - I}{\tau}\right) u(t) \right)
    &= \frac{\lambda |u|^{p-1}u}{\tau}\left( \left( \frac{1}{1-(p-1)\lambda\tau |u|^{p-1}} \right)^{\frac{1}{p-1}} -1 \right) \\
    &\qquad + \left( \frac{1}{1-(p-1)\tau |u|^{p-1}} \right)^{\frac{1}{p-1}+1} (p-1)\lambda^2 |u|^{2p-2} u \\
    &\qquad - \left( \frac{1}{1-(p-1)\tau |u|^{p-1}} \right)^{\frac{1}{p-1}+1} (p-1)\lambda |u|^{p-2} \frac{u}{|u|} (\nabla u)^2 \\
    &\qquad - \left( \frac{1}{1-(p-1)\lambda \tau |u|^{p-1}} \right)^{\frac{1}{p-1}+2} p(p-1)^2 \lambda^2 \tau |u|^{2p-3} \frac{u}{|u|} (\nabla u)^2 \\
    &\qquad - \left( \frac{1}{1-(p-1)\lambda\tau |u|^{p-1}} \right)^{\frac{1}{p-1}+1} (p-1)^2 \lambda |u|^{p-2} \frac{u}{|u|} (\nabla u)^2 .
    \end{aligned}
    $$
By \eqref{basic_diff_cal}, and since $(p-1)\tau |u|^{p-1} \leq 1/2$, the proof of this lemma is concluded.
\end{proof}

\begin{lem}\label{lem-3-1}
Let $p\geq 2$, $\frac{d(p-1)}{2} < q <\infty$ with $q\geq1$, and $r \in [q, \infty]$.
For any test function $\eta \in C_t^1 \big( (0,n\tau) ; C_{\bf x}^2 (\Omega) \big)$, the following estimate holds:
    \begin{align}
    &\left\| \int_0^{n\tau} S(n \tau- s) \eta (\cdot,s) ds
        - \tau \sum_{k=0}^{n-1} S(n \tau - k \tau) \eta (\cdot, k \tau) \right\|_{L^{r}(\Omega)} \nonumber\\
    &\qquad\leq \tau \int_{0}^{n\tau} (n \tau- t)^{-\frac{d}{2} \left(\frac{1}{q} - \frac{1}{r}\right) } \big\| (\partial_t -\Delta) \eta(\cdot,t) \big\|_{L^q(\Omega)}dt.\label{lemma_2_6_ineq1}
    \end{align}
\end{lem}

\begin{proof}
We use a similar argument as in Lemma 4.6 of \cite{I}.
The fundamental theorem of calculus gives
    $$
    \begin{aligned}
    &\int_0^{n\tau} S(n \tau- s) \eta(\cdot, s) ds  - \tau \sum_{k=0}^{n-1} S(n \tau - k \tau) \eta(\cdot, k\tau) \\
    &\qquad=\sum_{k=0}^{n-1} \int_{k\tau}^{(k+1)\tau} \Big( S(n \tau- s) \eta(\cdot, s) - S(n \tau - k
    \tau) \eta(\cdot, k \tau) \Big) ds \\
    &\qquad=\sum_{k=0}^{n-1} \int_{k\tau}^{(k+1)\tau} \int_{k\tau}^s \partial_t \big( S(n \tau- t)
    \eta(\cdot,t) \big) dtds=: Q.
    \end{aligned}
    $$
Since
    $$
    \begin{aligned}
    \partial_t \Big( S(n \tau- t) \eta(\cdot, t) \Big)^{\wedge}
    &= \partial_t \left( e^{(n \tau- t)|\xi|^2} \widehat{\eta}(\xi,t) \right)\\
    &= -|\xi|^2 e^{(n \tau- t)|\xi|^2} \widehat{\eta}(\xi,t) + e^{(n \tau- t)|\xi|^2}
    \partial_t \widehat{\eta}(\xi,t),
    \end{aligned}
    $$
we have the following identity:
    \begin{equation}\label{lemma_2_6_identity}
    \partial_t \Big( S(n \tau- t) \eta(\cdot,t) \Big) 
    = S(n \tau- t) \Big(( \partial_t -\Delta) \eta(\cdot,t)\Big).
    \end{equation}
Using the identity \eqref{lemma_2_6_identity} and integrating with respect to $s$, one yields
    $$
    \begin{aligned}
    Q&=\sum_{k=0}^{n-1} \int_{k\tau}^{(k+1)\tau} \int_{k\tau}^s  S(n \tau- t) \Big( (\partial_t -\Delta) \eta(\cdot,t) \Big) dtds \\
    &=\sum_{k=0}^{n-1} \int_{k\tau}^{(k+1)\tau} \left((k+1)\tau-t\right) S(n \tau- t) \Big( (\partial_t -\Delta) \eta(\cdot,t) \Big) dt \\
    &=\int_{0}^{n\tau} \left( \sum_{k=0}^{n-1} \chi_{(k\tau,(k+1)\tau)}(t)
    \left((k+1)\tau-t\right) \right) S(n \tau- t) \Big( (\partial_t -\Delta) \eta(\cdot,t) \Big) dt .
    \end{aligned}
    $$
Since
    $$
    \sum_{k=0}^{n-1} \chi_{(k\tau,(k+1)\tau)}(t) \left((k+1)\tau-t\right) \leq \tau\qquad\forall\,t\in(0,n\tau),
    $$
and by \eqref{linear_heat}, the quantity $Q$ is estimated as
    $$
    \begin{aligned}
    \|Q\|_{L^r(\Omega)}
    &=\left\| \int_{0}^{n\tau} \left( \sum_{k=0}^{n-1} \chi_{(k\tau,(k+1)\tau)}(t)
    \left((k+1)\tau-t\right) \right) S(n \tau- t) \Big( (\partial_t -\Delta) \eta(\cdot,t) \Big) dt
    \right\|_{L^{r}(\Omega)} \\
    &\leq \int_{0}^{n\tau} \tau \left\| S(n \tau- t) \Big( (\partial_t -\Delta) \eta(\cdot,t) \Big) \right\|_{L^{r}(\Omega)} dt \\
    &\leq \tau \int_{0}^{n\tau} (n \tau- t)^{-\frac{d}{2} \left(\frac{1}{q} - \frac{1}{r}\right) } \big\| (\partial_t -\Delta) \eta(\cdot,t) \big\|_{L^q(\Omega)} dt,
    \end{aligned}
    $$
and then the desired inequality \eqref{lemma_2_6_ineq1} follows.
\end{proof}

%%%%%%%%%%%%%%%%%%%%%%%%%%%%%%%%%%%%%%%%%%%%%%%%%%%%%%%%%%%%%%%%%%%%%%%%%%%%%%%%%%%%%%

\section{Proof of Theorem \ref{main_thm}: Convergence of $Z(n\tau)\phi$}\label{sec3}

In this section, we will prove Theorem \ref{main_thm} in the use of Theorem \ref{main_thm_2}.
Let $d\geq1$, $p \in [2,\infty)$, $q\geq 1$, $\frac{d(p-1)}{2}<q<\infty$ and $r\in[q,\infty]$.
Assume that $\phi\in W^{1,q}(\Omega)\cap L^{\infty}(\Omega)$.
To derive the main result \eqref{main_thm_est}, we now use an induction as follows:
For $n\geq 1$, we assume that
    \begin{equation}\label{induction_assume}
    (k\tau)^{\frac{d}{2}\left(\frac{1}{q}-\frac{1}{r}\right) - \left(1-\mu\right) } \| u(k\tau) - Z(k\tau)\phi \|_{L^r(\Omega)}
    \leq C_* \tau M_{p,q,\phi}\qquad\mbox{ for }k=0,\cdots,n-1,
    \end{equation}
where $C_*>0$ is a constant chosen later, and
    $$
    M_{p,q,\phi}
    := \| \phi\|_{L^q(\Omega)}^{p-2} \| \phi\|_{W^{1,q}(\Omega)}^{p+1}
    + \| \phi\|_{L^q(\Omega)}^{p-2} \| \phi\|_{W^{1,q}(\Omega)}^{2}
    + \tau \|\phi \|_{L^{\infty}(\Omega)}^{p-1}   \| \phi\|_{L^q(\Omega)}^{p-2} \| \phi\|_{W^{1,q}(\Omega)}^{2}.
    $$
From \eqref{eq-1-9} and \eqref{eq-1-8}, we have
    \begin{equation}\label{dif_decom}
    \begin{aligned}
    &(n\tau)^{\frac{d}{2}\left(\frac{1}{q}-\frac{1}{r}\right) - \left(1-\mu\right) } \| u(n\tau) - Z(n\tau)\phi \|_{L^r(\Omega)} \\
    &\qquad\leq (n\tau)^{\frac{d}{2}\left(\frac{1}{q}-\frac{1}{r}\right) - \left(1-\mu\right) }
    \left( \| Q_1 (n\tau) \|_{L^r(\Omega)} + \| Q_2 (n\tau)
    \|_{L^r(\Omega)} + \| Q_3 (n\tau) \|_{L^r(\Omega)} \right),
    \end{aligned}
    \end{equation}
where
    $$
    \begin{aligned}
    Q_1 (n\tau)
    &:= \int_0^{n\tau} S(n\tau -s) \left( \lambda|u|^p u(s) - \left(\frac{N(\tau) - 1}{\tau}\right) u (s) \right) ds
    ,\\
    Q_2 (n\tau)
    &:= \int_0^{n\tau} S(n\tau -s) \left( \frac{N(\tau) - 1}{\tau}\right) u(s) ds - \tau \sum_{0
    \leq k <n} S(n\tau- k\tau) \left( \frac{N(\tau)-I}{\tau}\right) u (k\tau) ,\\
    Q_3 (n\tau)
    &:= \tau \sum_{0 \leq k <n} S(n\tau -k\tau) \left[ \left( \frac{N(\tau)-I}{\tau}\right)
    u(k\tau) - \left( \frac{N(\tau) -I}{\tau}\right) Z(k\tau)\right].
    \end{aligned}
    $$
By Theorem \ref{QS-thm}, the quantity $Q_1(n\tau)$ is estimated to be
    \begin{align}
    \| Q_1 (n\tau) \|_{L^r(\Omega)}
    &\leq \int_0^{n\tau} \left\| S(n\tau -s) \left( \lambda|u(s)|^p u(s) - \left(\frac{N(\tau) - 1}{\tau}\right) u (s)
    \right)  \right\|_{L^{r}(\Omega)} ds \nonumber\\
    &\leq C_{d,p,q} \int_0^{n\tau} (n\tau -s)^{-\frac{d}{2}\left(\frac{1}{q}-\frac{1}{r}\right)} \left\| \lambda|u(s)|^p
    u(s) - \left(\frac{N(\tau) - 1}{\tau}\right) u (s)  \right\|_{L^{q}(\Omega)} ds.\label{Q1-estimate-ineq1}
    \end{align}
Using \eqref{basic_ineq_2} and Corollary \ref{precise_original}, one yields
    \begin{align}
    \left\| \lambda |u(s)|^p u(s) - \left(\frac{N(\tau) - 1}{\tau}\right) u (s) \right\|_{L^{q}(\Omega)}
    &\leq c_p \tau \left\| |u(s)|^{2p-1} \right\|_{L^{q}(\Omega)} \nonumber\\
    &\leq C_{d,p,q} \tau s^{-\mu} \| \phi\|_{L^q(\Omega)}^{p-2} \| \phi\|_{W^{1,q}(\Omega)}^{p+1}.\label{Q1-estimate-ineq2}
    \end{align}
By \eqref{Q1-estimate-ineq2}, the inequality \eqref{Q1-estimate-ineq1} implies
    \begin{align}
    &(n\tau)^{\frac{d}{2}\left(\frac{1}{q}-\frac{1}{r}\right) - \left(1-\mu\right)} \| Q_1 (n\tau) \|_{L^r(\Omega)} \nonumber\\
    &\qquad\leq C_{d,p,q} \tau (n\tau)^{\frac{d}{2}\left(\frac{1}{q}-\frac{1}{r}\right) - \left(1-\mu\right)} \| \phi\|_{L^q(\Omega)}^{p-2} \| \phi\|_{W^{1,q}(\Omega)}^{p+1}
    \int_0^{n\tau} (n\tau -s)^{-\frac{d}{2}\left(\frac{1}{q}-\frac{1}{r}\right)} s^{-\mu} ds \nonumber\\
    &\qquad\leq C_{d,p,q} \tau \| \phi\|_{L^q(\Omega)}^{p-2} \| \phi\|_{W^{1,q}(\Omega)}^{p+1} .\label{Q1-estimate-ineq3}
    \end{align}
Next, we consider the estimate for $Q_2(n\tau)$ in \eqref{dif_decom}.
Using Lemma \ref{lem-3-1} with $\eta(t) = \left( \frac{N(\tau) - 1}{\tau}\right) u(t)$ and Lemma \ref{lem-2-4}, we have
    \begin{align}
    \| Q_2 (n\tau) \|_{L^r(\Omega)}&\leq \tau \int_{0}^{n\tau} (n \tau- t)^{-\frac{d}{2}\left(\frac{1}{q}-\frac{1}{r}\right)} \left\| (\partial_t - \Delta) \left(
    \frac{N(\tau) - 1}{\tau}\right) u(t) \right\|_{L^q(\Omega)} dt \nonumber\\
    &\leq C_{p} \tau \int_{0}^{n\tau} (n \tau- t)^{-\frac{d}{2}\left(\frac{1}{q}-\frac{1}{r}\right)} A(t) dt,\label{Q2-estimate-ineq1}
    \end{align}
where
$$
A(t):=\left\| |u(t)|^{2p-1} \right\|_{L^q(\Omega)} + \left\| |u(t)|^{p-2} |\nabla u(t)|^2 \right\|_{L^q(\Omega)} + \tau \left\| |u(t)|^{2p-3} |\nabla u(t)|^2 \right\|_{L^q(\Omega)}.
$$
By Corollary \ref{precise_original}, one gets
    \begin{align}
    A(t) \leq C_{d,p,q} t^{-\mu} M_{p,q,\phi},
    \end{align}
so the inequality \eqref{Q2-estimate-ineq1} implies
    \begin{equation}\label{Q2-estimate-ineq2}
    (n\tau)^{\frac{d}{2}\left(\frac{1}{q}-\frac{1}{r}\right) - \left(1-\mu\right) } \| Q_2 (n\tau) \|_{L^r(\Omega)} \leq C_{d,p,q} \tau M_{p,q,\phi}.
    \end{equation}
Finally, we derive the estimate for $Q_3(n\tau)$ in \eqref{dif_decom}.
By \eqref{linear_heat} and Lemma \ref{lem-2-3}, one has
    \begin{align}
    &\| Q_3 (n\tau) \|_{L^r(\Omega)} \nonumber\\
    &\qquad\leq \tau \sum_{0 \leq k <n} \left\| S(n\tau -k\tau) \left[ \left( \frac{N(\tau)-I}{\tau}\right)
    u(k\tau) - \left( \frac{N(\tau) -I}{\tau}\right) Z(k\tau)\right]  \right\|_{L^{r}(\Omega)} \nonumber\\
    &\qquad\leq C_{d,p,q} \tau \sum_{0 \leq k <n} (n\tau - k\tau)^{-\frac{d}{2}\left(\frac{1}{q}-\frac{1}{r}\right)} \left\| \left( \frac{N(\tau)-I}{\tau}\right)
    u(k\tau) - \left( \frac{N(\tau) -I}{\tau}\right) Z(k\tau)  \right\|_{L^q(\Omega)} \nonumber\\
    &\qquad\leq C_{d,p,q} \tau \sum_{0 \leq k <n} (n\tau - k\tau)^{-\frac{d}{2}\left(\frac{1}{q}-\frac{1}{r}\right)}\left( \|u(k\tau)\|_{L^{\frac{qr(p-1)}{r-q}}(\Omega)}^{p-1} + \|Z (k\tau)\|_{L^{\frac{qr(p-1)}{r-q}}(\Omega)}^{p-1} \right) \nonumber\\
    &\qquad\qquad\qquad\qquad\qquad\times\|u (k\tau) - Z (k\tau)\|_{L^{r}(\Omega)}.\label{Q3-estimate-ineq1}
    \end{align}
From Theorem \ref{QS-thm} and Theorem \ref{main_thm_2}, one sees
    \begin{equation}\label{Q3-estimate-ineq2}
    \|u(k\tau)\|_{L^{\frac{qr(p-1)}{r-q}}(\Omega)}^{p-1} + \|Z (k\tau)\|_{L^{\frac{qr(p-1)}{r-q}}(\Omega)}^{p-1}
    \leq C_{d,p,q} (k\tau)^{-\frac{d}{2}\left(\frac{p-2}{q}+ \frac{1}{r}\right)} \| \phi\|_{L^q(\Omega)}^{p-1},
    \end{equation}
and on the other hand, the assumption \eqref{induction_assume} is rewritten as
    \begin{equation}\label{Q3-estimate-ineq3}
    \|u (k\tau) - Z (k\tau)\|_{L^{r}(\Omega)}
    \leq C_* \tau (k\tau)^{-\frac{d}{2}\left(\frac{1}{q}-\frac{1}{r}\right) + \left(1-\mu\right) } M_{p,q,\phi}\qquad\mbox{ for }k=0,\cdots,n-1.
    \end{equation}
By \eqref{Q3-estimate-ineq2} and \eqref{Q3-estimate-ineq3}, the inequality \eqref{Q3-estimate-ineq1} gives
    \begin{align}
    &(n\tau)^{\frac{d}{2}\left(\frac{1}{q}-\frac{1}{r}\right) - \mu} \| Q_3 (n\tau) \|_{L^r(\Omega)} \nonumber\\
    &\qquad\leq C_{d,p,q} C_* \tau^2 (n\tau)^{\frac{d}{2}\left(\frac{1}{q}-\frac{1}{r}\right) - \left(1-\mu\right) } \| \phi\|_{L^q(\Omega)}^{p-1} \|\phi\|_{W^{1,q}(\Omega)}^{2p-1} \nonumber\\
    &\qquad\qquad\times \sum_{0 \leq k <n} (n\tau - k\tau)^{-\frac{d}{2}\left(\frac{1}{q}-\frac{1}{r}\right)} (k\tau)^{1-\frac{d(p-1)}{q}} \nonumber\\
    &\qquad\leq C_{d,p,q} C_* \tau (n\tau)^{1-\frac{d(p-1)}{q}+\mu}  \| \phi\|_{L^q(\Omega)}^{p-1} M_{p,q,\phi} .\label{Q3-estimate-ineq4}
    \end{align}
As seen in \eqref{cal_assum_T0}, we note that since $(n\tau) <T_0$, and by \eqref{def_T0},
    \begin{equation}\label{Q3-estimate-ineq5}
    C_{d,p,q} (n\tau)^{1-\frac{d(p-1)}{q}+\mu} \|\phi\|_{L^q(\Omega)}^{p-1}
    \leq \frac{1}{2}.
    \end{equation}
Applying \eqref{Q3-estimate-ineq5} to \eqref{Q3-estimate-ineq4}, one gets
    \begin{equation}\label{Q3-estimate-ineq6}
    (n\tau)^{\frac{d}{2}\left(\frac{1}{q}-\frac{1}{r}\right) - \left(1-\mu\right) } \| Q_3 (n\tau) \|_{L^r(\Omega)} \\
    \leq \frac{1}{2} C_* \tau M_{p,q,\phi} .
    \end{equation}
From \eqref{Q1-estimate-ineq3}, \eqref{Q2-estimate-ineq2} and \eqref{Q3-estimate-ineq6}, the error $u(n\tau)-Z(n\tau)\phi$ in \eqref{dif_decom} is estimated to be
    $$
    (n\tau)^{\frac{d}{2}\left(\frac{1}{q}-\frac{1}{r}\right) - \left(1-\mu\right) } \| u(n\tau) - Z(n\tau)\phi \|_{L^r(\Omega)} \leq \left( C_{d,p,q} + \frac{1}{2} C_* \right) \tau M_{p,q,\phi}.
    $$
If the constant $C_*>0$ is chosen by $C_* \geq 2C_{d,p,q}$, then the error estimate \eqref{induction_assume} for $k=n$ holds true, so the proof of Theorem \ref{main_thm} is concluded.

%%%%%%%%%%%%%%%%%%%%%%%%%%%%%%%%%%%%%%%%%%%%%%%%%%%%%%%%%%%%%%%%%%%%%%%%%%%%

\section{Numerical experiments}\label{sec4}

In this section, we give some numerical results based on the operator splitting method \eqref{operator-splitting-approximation} and confirm the analyzed convergence rate in Theorem \ref{main_thm}.
For $0<h<1$, let $\mathcal{T}_h$ be a family of regular partitions of $\Omega$ into disjoint triangular elements.
The finite element spaces are defined by
$$
\begin{aligned}
V_h&=\{v\in C(\overline{\Omega}):v|_K \in P_1(K)\quad\forall K\in\mathcal{T}_h\},\\
V_h^0&=\{v\in V_h:v|_{\partial\Omega}=0\}\subset H_0^1(\Omega),
\end{aligned}
$$
where $P_1(K)$ means the space of linear functions defined on $K$.
Now, we propose the numerical scheme related to \eqref{operator-splitting-approximation} as follows: For a given switching time $\tau\ll 1$, let $u_{h,\tau}^n\in V_h^0$ be the approximation of $Z(n\tau)\phi$ for $n\in\mathbb{N}$. Firstly, we simulate some numerical experiments for three types of initial conditions on a square domain, and then we confirm the $L^2$-errors of discrete solution for a singular initial function on a non-convex polygon.
Moreover, we show a certain convergence rate on a three-dimensional cubic domain.
Lastly, we give some numerical results of the proposed method in computational time, which are compared with two existing algorithms.
\begin{tcolorbox}
\underline{{\it Algorithm A}}
\begin{enumerate}[{\bf 1.}]
 \item Set $n\leftarrow 0$ and $u_{h,\tau}^0=\pi_h\phi$,
 where $\pi_h$ denotes the interpolation operator on $V_h^0$.
 \item Set $n\leftarrow n+1$. Compute
 $$
 u_{h,\tau}^{n-1/2}=u_{h,\tau}^{n-1}\left(\frac{1}{1-(p-1)\lambda\tau\left|u_{h,\tau}^{n-1}\right|^{p-1}}\right)^{\frac{1}{p-1}}.
 $$
 \item Find $u_{h,\tau}^n\in V_h^0$ such that
 \begin{equation}\label{algorithmA-step3}
 \left(\frac{u_{h,\tau}^{n}-u_{h,\tau}^{n-1/2}}{\tau},v_h\right)+\left(\nabla u_{h,\tau}^n,\nabla v_h\right)=0\qquad\forall v_h\in V_h^0.
 \end{equation}
 \item Repeat 2-3 until $0<n\tau<T_2$, where $T_2>0$ is a time given by \eqref{def_T2}.
\end{enumerate}
\end{tcolorbox}

\vspace{0.3cm}

\noindent{\bf Example 1.}
In the first experiment, we simulate {\it Algorithm A} with various initial functions in the two-dimensional space.
We try to check the rate of convergence for $u_{h,\tau}^n$ on the square domain $\Omega=(0,1)^2$ depicted by Figure \ref{fig1}.
\begin{figure}[htbp]
\begin{center}
\subfigure[The domain $\Omega$]{
\includegraphics[height=5.5cm]{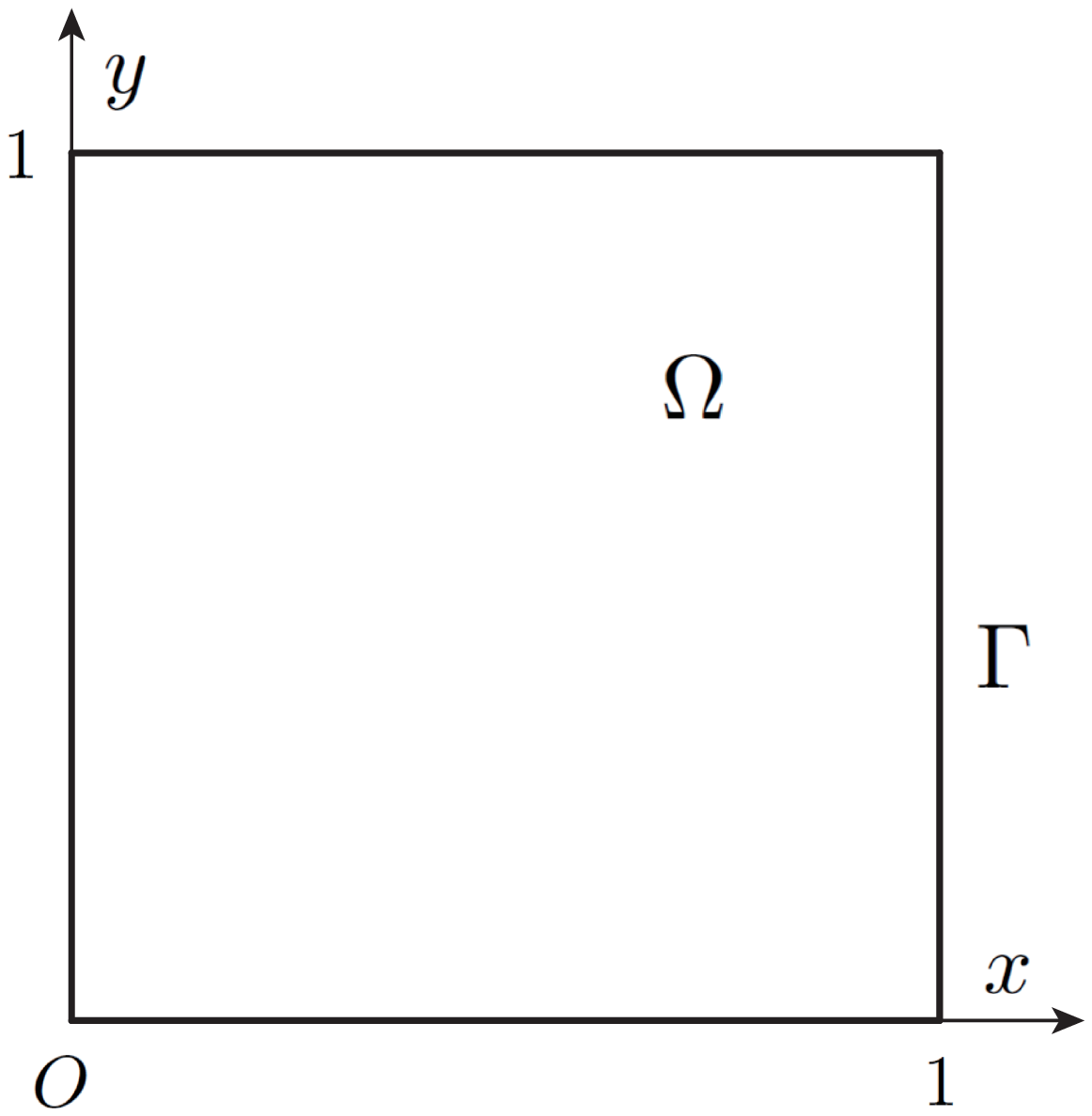}\label{ex1_1.fig}}
\qquad
\subfigure[The initial triangulation of $\Omega$]{
\includegraphics[height=5.5cm]{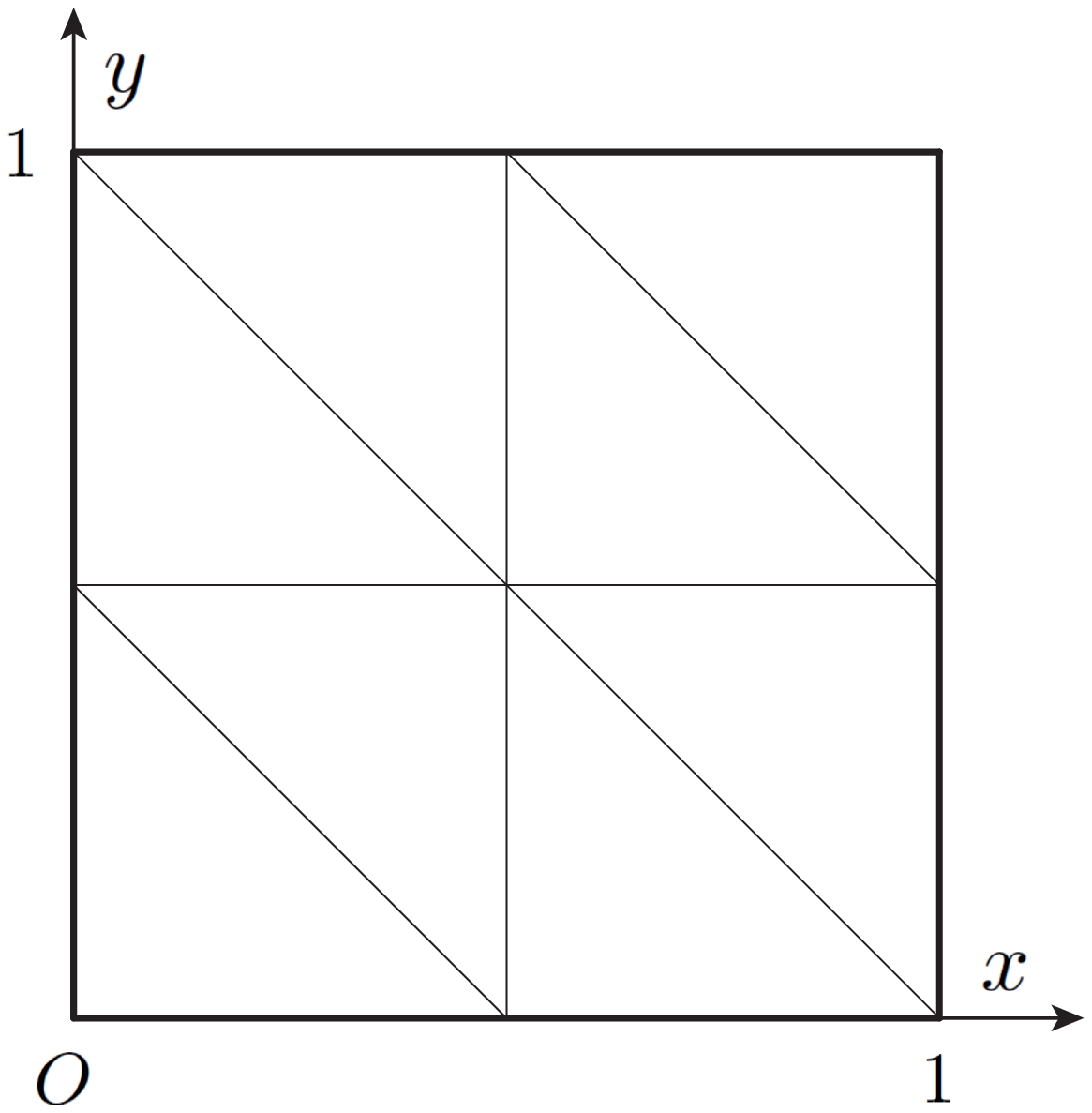}\label{ex1_2.fig}}
\end{center}
%\vspace{-0.5cm}
\caption{The computational domain (Example 1)}\label{fig1}
\end{figure}
Choose $p=5/2$ and $\lambda=1$ in all numerical experiments of this section.
In this example, the initial function $\phi=\phi_0$ is given by
\begin{equation}\label{example1-phi}
\phi_0:=\sin(\pi x)\sin(\pi y).
\end{equation}
Set the meshsize $h_j=2^{-j}$ on a level number $j\geq 1$ and the time step $\tau_k=2^{-k}$ on a level number $k\geq 1$.
Since the exact formula of $u$ satisfying \eqref{eq-main} can not be found, we define the sequential $L^2$-error for the time step $\tau_k$ on the fixed meshsize $h_j$ as follows:
$$
\mathcal{E}_u=\|u_{h_j,\tau_k}-u_{h_j,\tau_{k-1}}\|_{L^2(\Omega)},
$$
where $u_{h_j,\tau_k}$ denotes the discrete solution $u_{h_j,\tau_k}^{N_k}$ with $N_k:=0.1/\tau_k$, i.e., it means the approximation of $u(T)$ on the time $T=0.1$.
The convergence rate is defined by $Rate:=\log_2(e_{k-1}/e_k)$, provided that $e_k$ is the error on the $k$-th level.

\begin{table}[htbp]
\caption{The $L^2$-error $\mathcal{E}_u$ on the time $T=0.1$ by the initial function $\phi=\phi_0$}\label{tab1}
\begin{center} \footnotesize
\subtable[$h=2^{-7}$]{\label{tab1a}
\begin{tabular}{|c||c c|}
\hline
%\vspace{-0.15cm}
% & & \\
$N_k$ & \multicolumn{1}{c}{$\mathcal{E}_u$} & \multicolumn{1}{c|}{$Rate$}\\
\hline\hline
%\vspace{-0.15cm}
% & & \\
128&	1.4669E-3&	- \\
256&	7.4000E-4&	0.99 \\
512&	3.7167E-4&	0.99 \\
1024&	1.8625E-4&	1.00 \\
2048&	9.3232E-5&	1.00 \\
\hline
\end{tabular}
}
\subtable[$h=2^{-10}$]{\label{tab1b}
\begin{tabular}{|c||c c|}
\hline
%\vspace{-0.15cm}
% & & \\
$N_k$ & \multicolumn{1}{c}{$\mathcal{E}_u$} & \multicolumn{1}{c|}{$Rate$}\\
\hline\hline
%\vspace{-0.15cm}
% & & \\
128&	1.4668E-3&	- \\
256&	7.3998E-4&	0.99 \\
512&	3.7165E-4&	0.99 \\
1024&	1.8625E-4&	1.00 \\
2048&	9.3228E-5&	1.00 \\
\hline
\end{tabular}
}
\end{center}
\end{table}

On Table \ref{tab1}, we show the $L^2$-error $\mathcal{E}_u$ on the fixed meshsize $h=2^{-7}$ or $h=2^{-10}$.
If we assume that $\|Z(n\tau)\phi-u_{h,\tau}^n\|_{L^2(\Omega)}\leq C(\tau+h^2)$ which is perhaps expected by the $L^2$-error estimate for parabolic problem (cf. \cite{QVN,TG}), and by Theorem \ref{main_thm}, one sees
\begin{align}
\|u(n\tau)-u_{h,\tau}^n\|_{L^2(\Omega)}&\leq\|u(n\tau)-Z(n\tau)\phi\|_{L^2(\Omega)}+\|Z(n\tau)\phi-u_{h,\tau}^n\|_{L^2(\Omega)}\nonumber\\
&\leq C(\tau+h^2),\label{error-estimate}
\end{align}
where $C>0$ is a constant independent of $h$ and $\tau$.
From the estimate \eqref{error-estimate}, we expect that $\mathcal{E}_u=\mathcal{O}(\tau)$ for a sufficiently small $h$.
As seen in Table \ref{tab1}, it is confirmed that the $L^2$-error $\mathcal{E}_u$ has the expected convergence rate $1$ as the time step $\tau_k\rightarrow 0$.

In addition, we try to find the approximation $u_{h,\tau}^n$ obtained by {\it Algorithm A}, when the initial functions $\phi=\phi_i$ for $i=1,2$ are given by
\begin{align}
\phi_1&:=\left\{\begin{aligned}
&x\sin(\pi y)&&~~\,\mbox{ for }x\leq 1/2,\\
&(1-x)\sin(\pi y)&&~~\,\mbox{ for }x>1/2,
\end{aligned}\right.\label{example2-phi1}\\
\phi_2&:=\left\{\begin{aligned}
&x\sin(\pi y)&&\mbox{ for }x\leq 1/2,\\
&2(1-x)\sin(\pi y)&&\mbox{ for }x>1/2.
\end{aligned}\right.\label{example2-phi2}
\end{align}
\begin{table}[htbp]
\caption{The $L^2$-error $\mathcal{E}_u$ on the time $T=0.1$, obtained by $\phi=\phi_i$ for $i=1,2$}\label{tab2}
\begin{center} \footnotesize
\subtable[$h=2^{-7}$]{\label{tab2a}
\begin{tabular}{|c||c c|c c|}
\hline
%\vspace{-0.15cm}
% & & \\
$N_k$ & \multicolumn{1}{c}{$\mathcal{E}_u$ ($\phi=\phi_1$)} & \multicolumn{1}{c|}{$Rate$} & \multicolumn{1}{c}{$\mathcal{E}_u$ ($\phi=\phi_2$)} & \multicolumn{1}{c|}{$Rate$}\\
\hline\hline
%\vspace{-0.15cm}
% & & \\
128&	5.7360E-4&	- &    9.2994E-4&	- \\
256&	2.8944E-4&	0.99&  4.6872E-4&	0.99 \\
512&	1.4539E-4&	0.99&  2.3531E-4&	0.99 \\
1024&	7.2862E-5&	1.00&  1.1789E-4&	1.00 \\
2048&	3.6473E-5&	1.00&  5.9002E-5&	1.00 \\
\hline
\end{tabular}
}
\subtable[$h=2^{-10}$]{\label{tab2b}
\begin{tabular}{|c||c c|c c|}
\hline
%\vspace{-0.15cm}
% & & \\
$N_k$ & \multicolumn{1}{c}{$\mathcal{E}_u$ ($\phi=\phi_1$)} & \multicolumn{1}{c|}{$Rate$} & \multicolumn{1}{c}{$\mathcal{E}_u$ ($\phi=\phi_2$)} & \multicolumn{1}{c|}{$Rate$}\\
\hline\hline
%\vspace{-0.15cm}
% & & \\
128&	5.7355E-4&	-&     9.3508E-4&	- \\
256&	2.8941E-4&	0.99&  4.7131E-4&	0.99 \\
512&	1.4537E-4&	0.99&  2.3660E-4&	0.99 \\
1024&	7.2855E-5&	1.00&  1.1854E-4&	1.00 \\
2048&	3.6470E-5&	1.00&  5.9326E-5&	1.00 \\
\hline
\end{tabular}
}
\end{center}
\end{table}
Compared with the initial function $\phi=\phi_0$ of \eqref{example1-phi}, we notice that two functions $\phi_1$ and $\phi_2$ lose the smoothness and furthermore, the function $\phi_2$ is even discontinuous at $x=1/2$.
Nevertheless, Table \ref{tab2} describes that each convergence rate of the $L^2$-error $\mathcal{E}_u$ corresponding to the both initial functions $\phi_1$ and $\phi_2$ is identical to the predicted value $1$.

\vspace{0.3cm}

\noindent{\bf Example 2.}
In the second experiment, we give the $L^2$-errors of discrete solution obtained by {\it Algorithm A}, when the initial function has a corner singularity near a re-entrant corner.
Let $\Omega=((-1,1)\times(-1,1))\setminus([-1,0]\times[-1,0])$ be the L-shaped domain which is depicted in Figure \ref{fig2}.
\begin{figure}[htbp]
\begin{center}
\subfigure[The domain $\Omega$]{
\includegraphics[height=5.5cm]{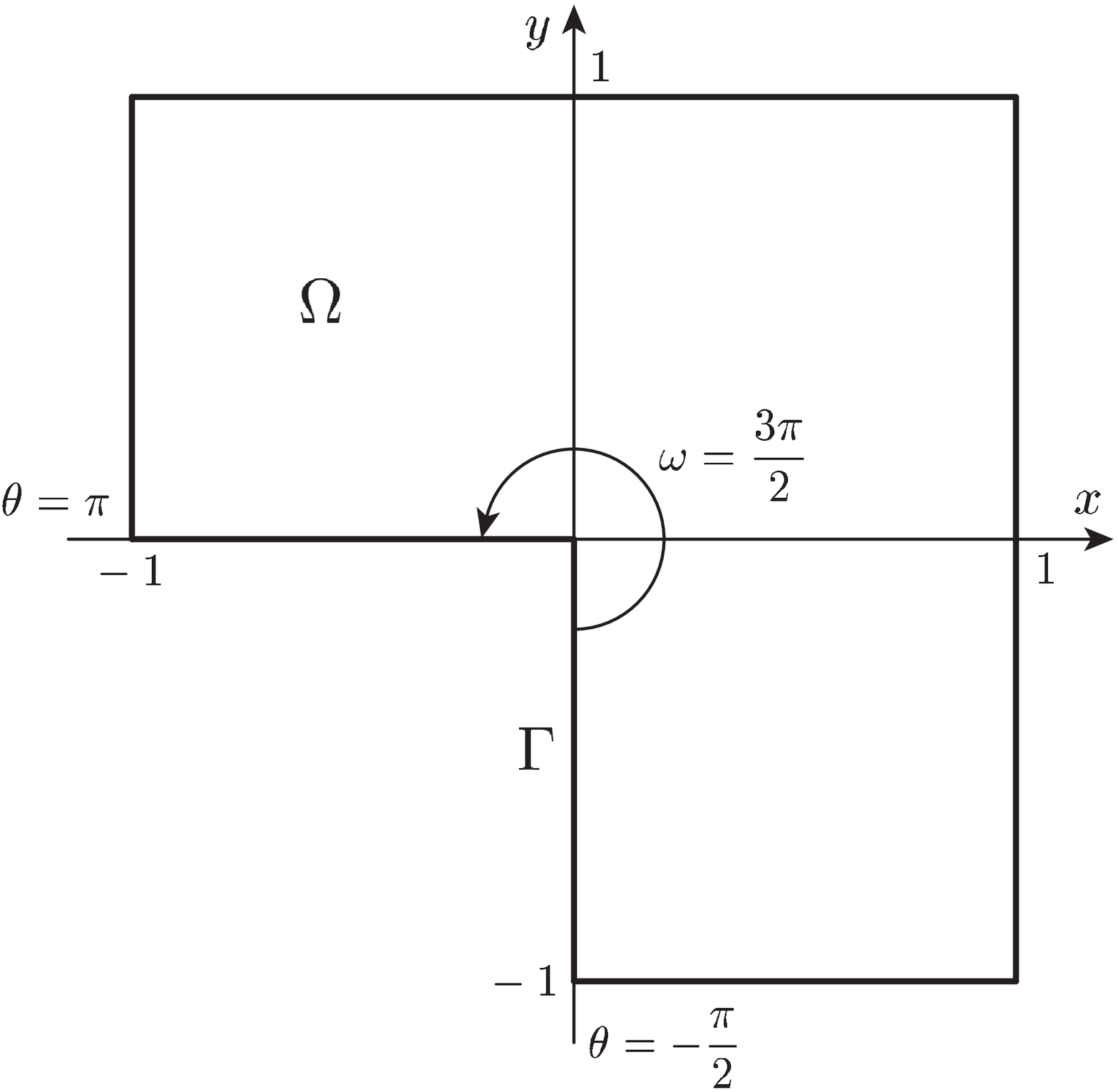}\label{ex2_1.fig}}
\qquad
\subfigure[The initial triangulation of $\Omega$]{
\includegraphics[height=5.5cm]{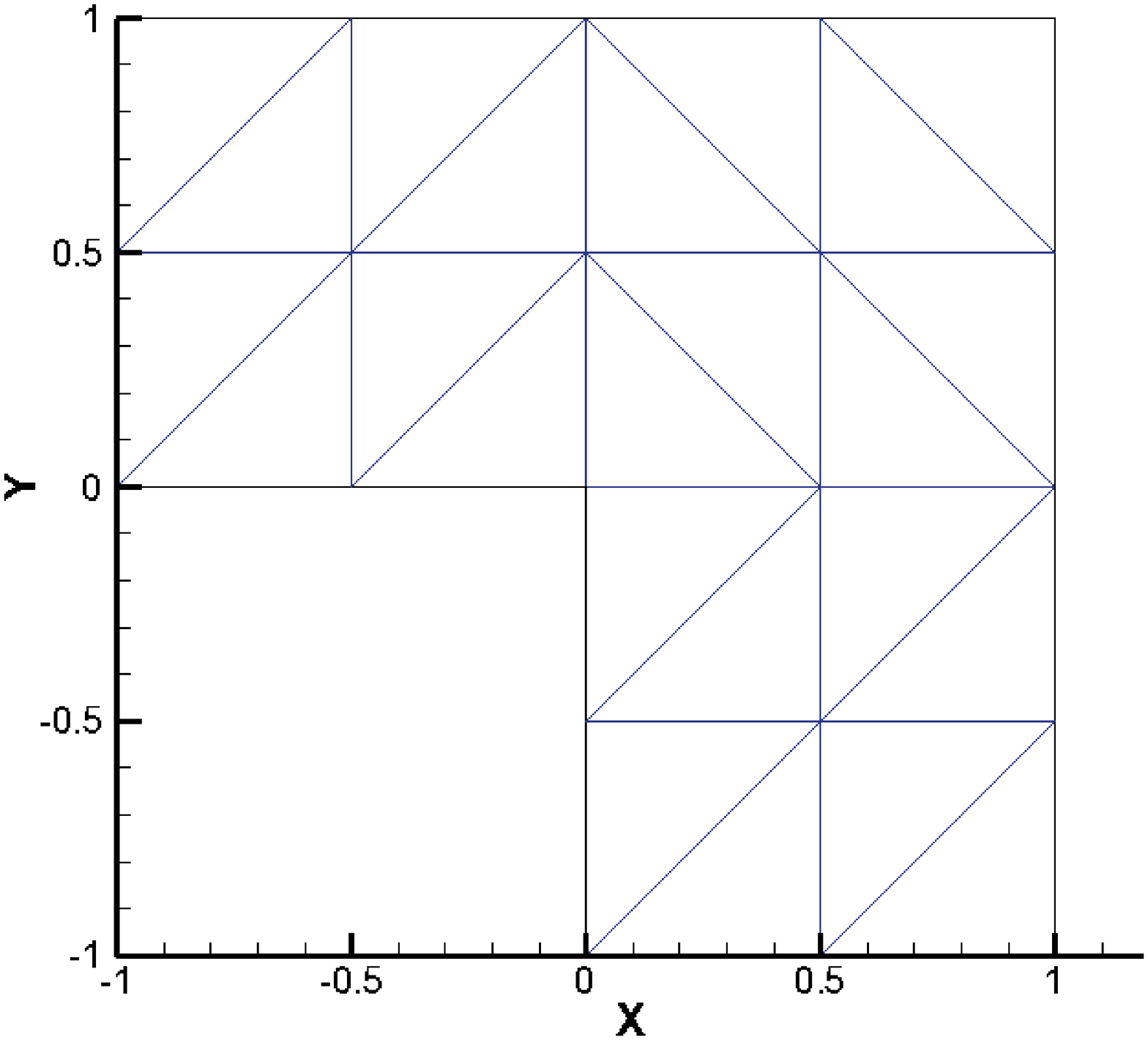}\label{ex2_2.fig}}
\end{center}
%\vspace{-0.5cm}
\caption{The computational domain (Example 2)}\label{fig2}
\end{figure}
The non-convex vertex is located at the origin whose opening angle is $\omega=\omega_2-\omega_1=3\pi/2$ with $\omega_1=-\pi/2$ and $\omega_2=\pi$.
Let $(r,\theta)$ be the polar coordinate with $r=\sqrt{x^2+y^2}$ and $\theta\in(\omega_1,\omega_2)$.
The cutoff function $\chi(r)\in C^2(\mathbb{R}^+)$ is defined by
$$
\chi(r)=\left\{\begin{aligned}
&1 && \mbox{ for }r\leq 1/4,\\
&-192r^5+480r^4-440r^3+180r^2-33.75r+3.375 && \mbox{ for }1/4<r<3/4,\\
&0 && \mbox{ for }r\geq 3/4.
\end{aligned}\right.
$$
With $\alpha:=\pi/\omega=2/3$, the initial function $\phi$ is set to be
\begin{equation}\label{sing-initial}
\phi=\chi(r)r^{\alpha}\sin[\alpha(\theta-\omega_1)],
\end{equation}
where it has the corner singularity of the Laplace operator with the Dirichlet boundary condition (cf. \cite{GE}) and also vanishes on the boundary $\partial\Omega$.
Actually, a direct calculation gives that
\begin{equation}\label{regularity-init}
\begin{aligned}
&\phi\in W^{1,q}(\Omega)&&\mbox{ for }~\frac{3}{2}<q<6,\\
&\phi\not\in W^{2,q}(\Omega)&&\mbox{ for }~\frac{3}{2}<q<\infty.
\end{aligned}
\end{equation}
The regularity result \eqref{regularity-init} implies that the initial condition $\phi$ in \eqref{sing-initial} guarantees Theorem \ref{main_thm}.

\begin{table}[htbp]
\caption{The $L^2$-error $\mathcal{E}_u$ for the singular initial function \eqref{sing-initial} on the L-shaped domain}\label{tab3}
\begin{center} \footnotesize
\subtable[$T=0.1$]{\label{tab3a}
\begin{tabular}{|c||c c|}
\hline
%\vspace{-0.15cm}
% & & \\
$N_k$ & \multicolumn{1}{c}{$\mathcal{E}_u$} & \multicolumn{1}{c|}{$Rate$}\\
\hline\hline
%\vspace{-0.15cm}
% & & \\
128&	6.4857E-4&	- \\
256&	3.2383E-4&	1.00 \\
512&	1.6178E-4&	1.00 \\
1024&	8.0857E-5&	1.00 \\
2048&	4.0419E-5&	1.00 \\
\hline
\end{tabular}
}
\subtable[$T=0.5$]{\label{tab3b}
\begin{tabular}{|c||c c|}
\hline
%\vspace{-0.15cm}
% & & \\
$N_k$ & \multicolumn{1}{c}{$\mathcal{E}_u$} & \multicolumn{1}{c|}{$Rate$}\\
\hline\hline
%\vspace{-0.15cm}
% & & \\
128&	1.4984E-4&	- \\
256&	7.2273E-5&	1.05 \\
512&	3.5465E-5&	1.03 \\
1024&	1.7564E-5&	1.01 \\
2048&	8.7395E-6&	1.01 \\
\hline
\end{tabular}
}
\subtable[$T=1.0$]{\label{tab3c}
\begin{tabular}{|c||c c|}
\hline
%\vspace{-0.15cm}
% & & \\
$N_k$ & \multicolumn{1}{c}{$\mathcal{E}_u$} & \multicolumn{1}{c|}{$Rate$}\\
\hline\hline
%\vspace{-0.15cm}
% & & \\
128&	2.3088E-5&	- \\
256&	1.0249E-5&	1.17 \\
512&	4.8115E-6&	1.09 \\
1024&	2.3287E-6&	1.05 \\
2048&	1.1452E-6&	1.02 \\
\hline
\end{tabular}
}
\end{center}
\end{table}
We try to check the convergence rates of the $L^2$-errors of discrete solution by {\it Algorithm A} on the L-shaped domain.
On Table \ref{tab3}, we describe the errors of $\mathcal{E}_u$ and their rates on the fixed meshsize $h=2^{-6}$ when $T=0.1$, $0.5$, and $1.0$.
Although the initial condition has a singularity near a re-entrant corner of non-convex polygon, it is shown that the convergence rates of $\mathcal{E}_u$ are almost identical to the analyzed value $1$.

\vspace{0.3cm}

\noindent{\bf Example 3.}
In the third experiment, we show the numerical results simulated by {\it Algorithm A} on the cubic domain $\Omega=(0,1)^3$ in the three-dimensional space.
Here, the triangulation of $\Omega$ consists of tetrahedrons with the meshsize $h$.
The initial data $\phi=\varphi_i$ for $i=0,1,2$ are chosen as
\begin{align}
\varphi_0&:=\sin(\pi x)\sin(\pi y)\sin(\pi z),\label{varphi0}\\
\varphi_1&:=\left\{\begin{aligned}
&x\sin(\pi y)\sin(\pi z)&&~~\,\mbox{ for }x\leq 1/2,\\
&(1-x)\sin(\pi y)\sin(\pi z)&&~~\,\mbox{ for }x>1/2,
\end{aligned}\right.\label{varphi1}\\
\varphi_2&:=\left\{\begin{aligned}
&x\sin(\pi y)\sin(\pi z)&&\mbox{ for }x\leq 1/2,\\
&2(1-x)\sin(\pi y)\sin(\pi z)&&\mbox{ for }x>1/2.
\end{aligned}\right.\label{varphi2}
\end{align}

\begin{table}[htbp]
\caption{The $L^2$-error $\mathcal{E}_u$ on the time $T=0.1$ in the cubic domain with $h=2^{-6}$}\label{tab4}
\begin{center} \footnotesize
\subtable[$\phi=\varphi_0$]{\label{tab4a}
\begin{tabular}{|c||c c|}
\hline
%\vspace{-0.15cm}
% & & \\
$N_k$ & \multicolumn{1}{c}{$\mathcal{E}_u$} & \multicolumn{1}{c|}{$Rate$}\\
\hline\hline
%\vspace{-0.15cm}
% & & \\
128&	1.3857E-3&	- \\
256&	6.9957E-4&	0.99 \\
512&	3.5148E-4&	0.99 \\
1024&	1.7617E-4&	1.00 \\
2048&	8.8192E-5&	1.00 \\
\hline
\end{tabular}
}
\subtable[$\phi=\varphi_1$]{\label{tab4b}
\begin{tabular}{|c||c c|}
\hline
%\vspace{-0.15cm}
% & & \\
$N_k$ & \multicolumn{1}{c}{$\mathcal{E}_u$} & \multicolumn{1}{c|}{$Rate$}\\
\hline\hline
%\vspace{-0.15cm}
% & & \\
128&	5.5141E-4&	- \\
256&	2.7846E-4&	0.99 \\
512&	1.3992E-4&	0.99 \\
1024&	7.0138E-5&	1.00 \\
2048&	3.5113E-5&	1.00 \\
\hline
\end{tabular}
}
\subtable[$\phi=\varphi_2$]{\label{tab4c}
\begin{tabular}{|c||c c|}
\hline
%\vspace{-0.15cm}
% & & \\
$N_k$ & \multicolumn{1}{c}{$\mathcal{E}_u$} & \multicolumn{1}{c|}{$Rate$}\\
\hline\hline
%\vspace{-0.15cm}
% & & \\
128&	8.4801E-4&	- \\
256&	4.2780E-4&	0.99 \\
512&	2.1485E-4&	0.99 \\
1024&	1.0767E-4&	1.00 \\
2048&	5.3893E-5&	1.00 \\
\hline
\end{tabular}
}
\end{center}
\end{table}

On Table \ref{tab4}, we give the $L^2$-errors on $T=0.1$ and their convergence rates in the various initial conditions $\phi=\varphi_i$ for $i=0,1,2$.
As seen in previous examples, the $L^2$-errors in the three-dimensional domain also decay to zero with the analyzed rate $1$.

\vspace{0.3cm}

\noindent{\bf Example 4.}
In the fourth experiment, we check the computational time of {\it Algorithm A} which is compared with two other types based on the implicit Euler method.
Precisely, the first comparing algorithm ({\it Type 1}\,) is given by
\begin{equation}\label{eq-4-1}
\frac{u^n -u^{n-1}}{\tau} = \Delta u^n + \lambda |u^n|^{p-1} u^n,
\end{equation}
and the second comparing algorithm ({\it Type 2}\,) is considered as
\begin{equation}\label{eq-4-2}
\frac{u^n -u^{n-1}}{\tau} = \Delta u^n + \lambda |u^{n-1}|^{p-1} u^n.
\end{equation}
Here, the nonlinearity in the equation \eqref{eq-4-1} is solved by the Picard iteration.
To check the performance of three numerical methods including {\it Algorithm A}, {\it Type 1} and {\it Type 2}, we try to measure the computation time to find the corresponding discrete solutions on $T=0.1$ with the fixed meshsize $h=2^{-7}$ by using three numerical methods in the test cases $\phi=\phi_0$, $\phi_1$ and $\phi_2$ defined in Example 1.

On Figure \ref{fig3}, we plot the graphs of the computation time for three algorithms (cf. \cite{HSI}).
As seen in each graph, we notice that {\it Algorithm A} is much faster than two other schemes in order to obtain the same $L^2$-errors.
The reason is perhaps due to the equality of stiffness matrices used in \eqref{algorithmA-step3} for each iteration $n$.

\begin{figure}[htbp]
\begin{center}
\subfigure[$\phi=\phi_0$]{
\includegraphics[height=4.4cm]{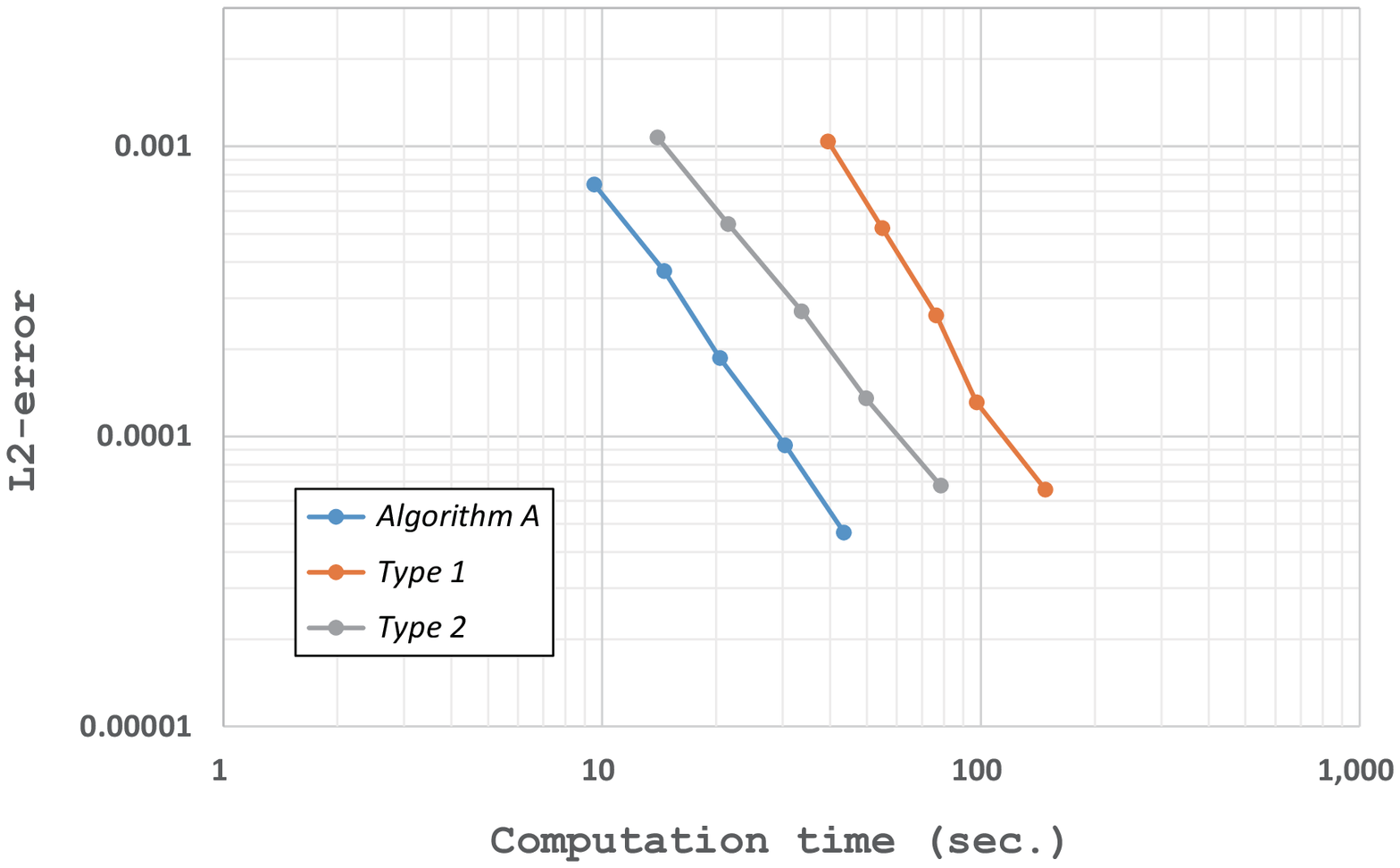}\label{comparison_ex1.fig}}
\quad
\subfigure[$\phi=\phi_1$]{
\includegraphics[height=4.4cm]{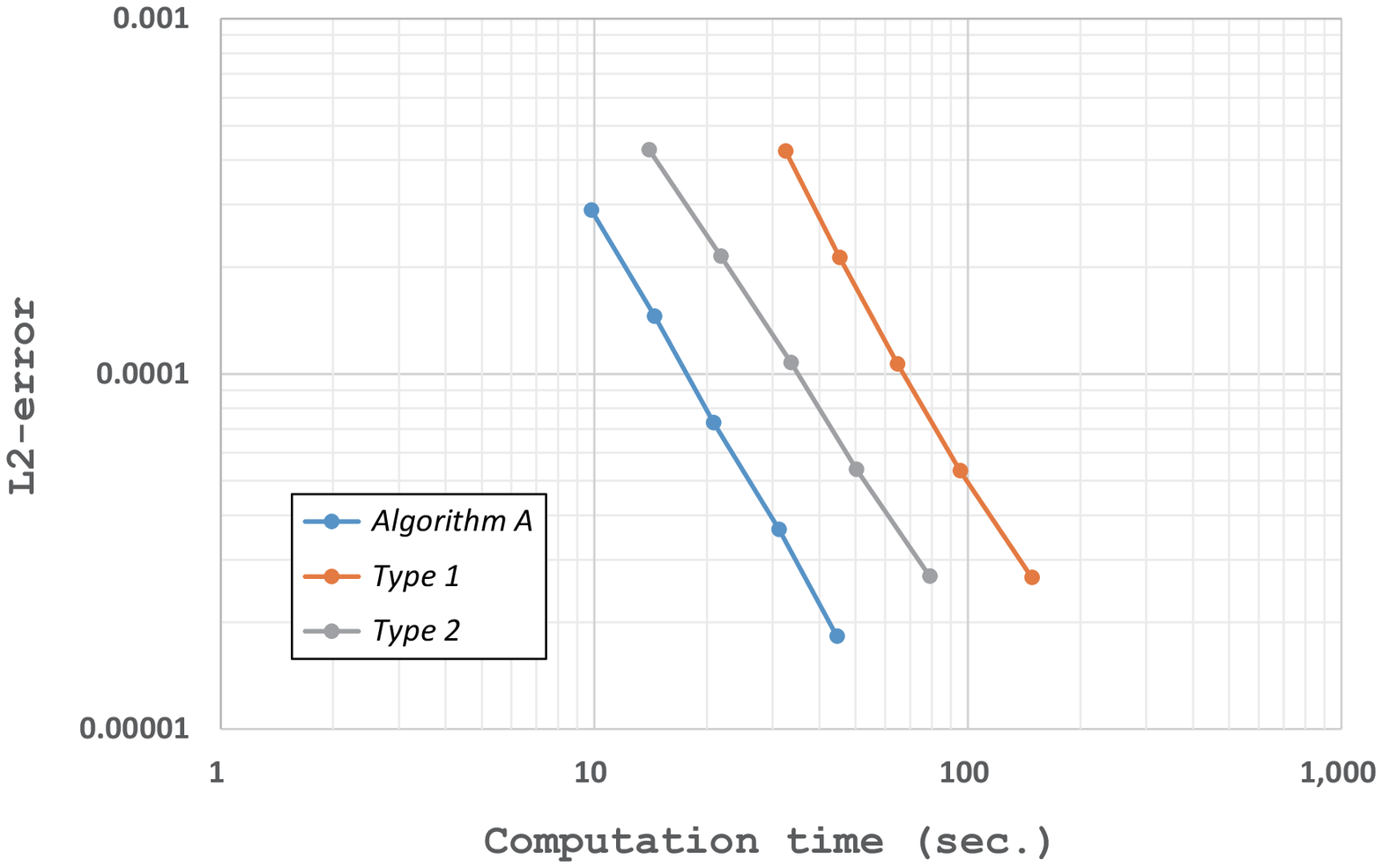}\label{comparison_ex2.fig}}
\subfigure[$\phi=\phi_2$]{
\includegraphics[height=4.4cm]{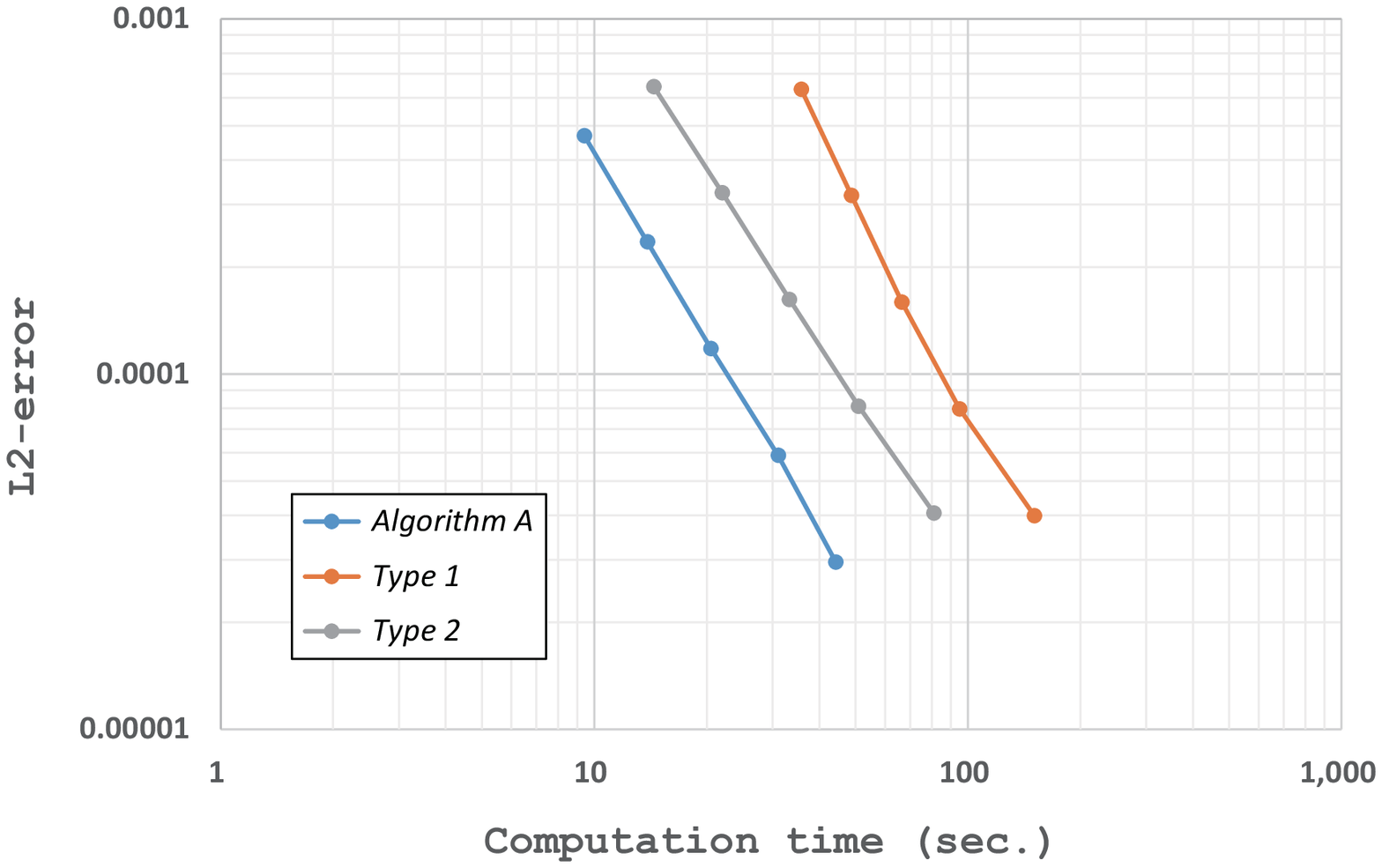}\label{comparison_ex3.fig}}
\end{center}
\vspace{-0.5cm}
\caption{Graphs related to the computation time of three numerical methods including {\it Algorithm A}, {\it Type 1} and {\it Type 2\,} for the initial conditions $\phi=\phi_0$, $\phi_1$ and $\phi_2$}\label{fig3}
\end{figure}

\section*{Conclusion}
In this paper, we analyze the convergence of the operator splitting method of the nonlinear heat equation with the Dirichlet boundary condition.
By various numerical experiments, it is confirmed that the $L^2$-error of operator splitting scheme has the analyzed convergence rate $1$ with respect to time step, and furthermore it is more effective in computation time than others.
In future works, we will extend our result to the rough initial data and develop a novel splitting scheme combined with the corner singularity expansion on a non-convex polygon (cf. \cite{CKT}).

%%%%%%%%%%%%%%%%%%%%%%%%%%%%%%%%%%%%%%%%%%%%%%%%%%%%%%%%%%%%%%%%%%%%%%%%%%%%

\section*{Appendix: The well-posedness of the approximation $Z(n\tau)\phi$}

We give the proof of Theorem \ref{main_thm_2} regarding the well-posedness of $Z(n\tau)\phi$ given in \eqref{operator-splitting-approximation}.
This proof is essentially similar to the one of Theorem \ref{QS-thm} which is shown in Theorem 15.2 of \cite{QS}.

Let $d\geq1$, $p>1$ and $q\geq1$ with $\frac{d(p-1)}{2}<q<\infty$.
From Proposition \ref{prop_def_N}, $Z(n\tau)\phi$ is well-defined for $n\tau \in (0,T_2)$.
So we consider the following set:
    $$
    \Lambda(d,p,q) := \left\{ N \in \mathbb{N} ~:~
    \max_{1\leq n\leq N} (n\tau)^{\frac{d}{2} \left(\frac{1}{q}-\frac{1}{r}\right)} \|Z(n\tau)\phi\|_{L^r(\Omega)}
    \leq C_*\|\phi\|_{L^q(\Omega)}  ~\mbox{for all}~ r\in[q,\infty]~ \right\},
    $$
where $C_*:= 4\max\{C_{d,p,q},1\}$ for a constant $C_{d,p,q}>0$ used in \eqref{NLH_basic}.
To derive Theorem \ref{main_thm_2}, it is sufficient to show that
$$
\Lambda(d,p,q)\supset\{n\in\mathbb{N}:n\tau< T_2\}.
$$

\vspace{0.3cm}

\noindent\underline{\it Step 1 (Base case)}.
We first consider the estimate of $Z(\tau)\phi$.
By \eqref{linear_heat} and the explicit form \eqref{NL-flow_2} of $N(t)$, and since $\tau \leq T_2 /2 \leq T_1 /2$, we have
    $$
    \begin{aligned}
    \tau^{\frac{d}{2} \left(\frac{1}{q}-\frac{1}{r}\right)} \|Z(\tau)\phi \|_{L^r(\Omega)}
    &= \tau^{\frac{d}{2} \left(\frac{1}{q}-\frac{1}{r}\right)} \|S(\tau)N(\tau)\phi \|_{L^r(\Omega)} \\
    &\leq (4\pi)^{-\frac{d}{2} \left(\frac{1}{q}-\frac{1}{r}\right)} \|N(\tau)\phi \|_{L^q(\Omega)} \\
    &\leq 2^{\frac{1}{p-1}} \|\phi \|_{L^q(\Omega)},
    \end{aligned}
    $$
and then this inequality implies that $1 \in \Lambda(d,p,q)$.

\noindent\underline{\it Step 2 (Inductive step)}.
Let $n\in\mathbb{N}$ be given with $n\tau <T_2$.
It is assumed that $n-1 \in \Lambda(d,p,q)$.
By the expression \eqref{eq-1-9} and \eqref{linear_heat}, one yields
    \begin{align}
    (n\tau)^{\frac{d}{2} \left(\frac{1}{q}-\frac{1}{r}\right)} \|Z(n\tau) \phi \|_{L^r(\Omega)}
    &\leq(n\tau)^{\frac{d}{2} \left(\frac{1}{q}-\frac{1}{r}\right)} \left\| S(n\tau) \phi \right\|_{L^r(\Omega)}
    + (n\tau)^{\frac{d}{2} \left(\frac{1}{q}-\frac{1}{r}\right)} B(n,\tau)\nonumber\\
    &\leq C_{d,q} \|\phi\|_{L^q(\Omega)} + (n\tau)^{\frac{d}{2} \left(\frac{1}{q}-\frac{1}{r}\right)} B(n,\tau),\label{appendix-ineq1}
    \end{align}
where $C_{d,q}:=(4\pi)^{-\frac{d}{2}\left(\frac{1}{q}-\frac{1}{r}\right)}$ and
    $$
    B(n,\tau):=\tau \sum_{0 \leq k <n} \left\|
    S(n\tau - k \tau) \left( \frac{N(\tau)-I}{\tau} \right) Z(k \tau) \phi \right\|_{L^r(\Omega)}.
    $$
Again, using \eqref{linear_heat}, we have
    $$
    \begin{aligned}
    B(n,\tau)&\leq \tau \sum_{0 \leq k <n} (n\tau - k\tau)^{-\frac{d}{2} \left(\frac{1}{q}-\frac{1}{r}\right)} \left\|
    \left(\frac{N(\tau) - I}{\tau}\right) Z(k\tau) \phi \right\|_{L^q(\Omega)} \\
    &\leq \tau \sum_{0 \leq k <n} (n\tau -
    k\tau)^{-\frac{d}{2} \left(\frac{1}{q}-\frac{1}{r}\right)} \|Z(k\tau) \phi \|_{L^{pq}(\Omega)}^{p} \\
    &= \tau \sum_{ 0 \leq k < n} (n\tau - k\tau)^{-\frac{d}{2} \left(\frac{1}{q}-\frac{1}{r}\right)} (k\tau)^{-\frac{dp}{2} \left(\frac{1}{q}-\frac{1}{pq}\right)}
    \left( (k\tau)^{\frac{d}{2} \left(\frac{1}{q}-\frac{1}{pq}\right)} \|Z(k \tau) \phi \|_{L^{pq}(\Omega)} \right)^{p} \\
    &\leq \tau \sum_{ 0 \leq k < n} (n\tau - k\tau)^{-\frac{d}{2} \left(\frac{1}{q}-\frac{1}{r}\right)} (k\tau)^{-\frac{dp}{2} \left(\frac{1}{q}-\frac{1}{pq}\right)}
    \left( \max_{0\leq k <n} (k\tau)^{\frac{d}{2} \left(\frac{1}{q}-\frac{1}{pq}\right)} \|Z(k \tau) \phi \|_{L^{pq}(\Omega)} \right)^{p} \\
    &\leq C_{d,p,q} (n\tau)^{-\frac{d}{2} \left(\frac{1}{q}-\frac{1}{r}\right) -\frac{dp}{2} \left(\frac{1}{q}-\frac{1}{pq}\right) +1}
    \left( \max_{0\leq k <n} (k\tau)^{\frac{d}{2} \left(\frac{1}{q}-\frac{1}{pq}\right)} \|Z(k \tau) \phi \|_{L^{pq}(\Omega)} \right)^{p}.
    \end{aligned}
    $$
From the inductive hypothesis, one sees
    $$
    \max_{0\leq k <n} (k\tau)^{\frac{d}{2} \left(\frac{1}{q}-\frac{1}{pq}\right)} \|Z(k \tau) \phi \|_{L^{pq}(\Omega)}
    \leq C_* \|\phi\|_{L^q(\Omega)},
    $$
and then we have
    \begin{equation}\label{appendix-ineq2}
    B(n,\tau)\leq C_{d,p,q} (n\tau)^{-\frac{d}{2} \left(\frac{1}{q}-\frac{1}{r}\right) -\frac{dp}{2} \left(\frac{1}{q}-\frac{1}{pq}\right) +1} C_*^p \|\phi\|_{L^q(\Omega)}^p.
    \end{equation}
By \eqref{appendix-ineq2}, the inequality \eqref{appendix-ineq1} becomes
    \begin{equation}\label{appendix-ineq4}
    (n\tau)^{\frac{d}{2} \left(\frac{1}{q}-\frac{1}{r}\right)} \|Z(n\tau) \phi \|_{L^r(\Omega)}
    \leq C_{d,q} \|\phi\|_{L^q(\Omega)} + C_{d,p,q} (n\tau)^{ \frac{1}{q} \left(q - \frac{d(p-1)}{2}\right) }
    C_*^p \|\phi\|_{L^q(\Omega)}^p.
    \end{equation}
Since
$$
n\tau < T_2 \leq T_0 = c_{d,p,q}\left( \left( 1/~ \|\phi\|_{L^q(\Omega)}^{(p-1)q} \right)^{\frac{1}{q - \frac{d(p-1)}{2}} } \right),
$$
and if $c_{d,p,q}>0$ is a sufficiently small constant satisfying $C_{d,p,q}~ c_{d,p,q}^{ \frac{1}{q} \left(q - \frac{d(p-1)}{2}\right) } C_*^{p-1}\leq 1/2$, then it is noted that
    \begin{align}
    C_{d,p,q} (n\tau)^{ \frac{1}{q} \left(q - \frac{d(p-1)}{2}\right) } C_*^p \|\phi\|_{L^q(\Omega)}^p
    &\leq C_{d,p,q}~ c_{d,p,q}^{ \frac{1}{q} \left(q - \frac{d(p-1)}{2}\right) } C_*^p \|\phi\|_{L^q(\Omega)} \nonumber\\
    &\leq \frac{1}{2} C_* \|\phi\|_{L^q(\Omega)}.\label{appendix-ineq3}
    \end{align}
Using \eqref{appendix-ineq3}, the inequality \eqref{appendix-ineq4} becomes
    $$
    \begin{aligned}
    (n\tau)^{\frac{d}{2} \left(\frac{1}{q}-\frac{1}{r}\right)} \|Z(n\tau) \phi \|_{L^r(\Omega)}
    &\leq \left( C_{d,q} + \frac{1}{2} C_* \right) \|\phi\|_{L^q(\Omega)} \\
    &\leq C_* \|\phi\|_{L^q(\Omega)} ,
    \end{aligned}
    $$
which implies that $n \in \Lambda(d,p,q)$.
As shown in {\it Step 1} and {\it 2}, the proof of Theorem \ref{main_thm_2} is concluded.

%%%%%%%%%%%%%%%%%%%%%%%%%%%%%%%%%%%%%%%%%%%%%%%%%%%%%%%%%%%%%%%%%%%%%%%%%%%%%%

\section*{Acknowledgements}\thispagestyle{empty}
This work was supported by the NRF (Republic of Korea) grant No. 2019R1C1C1008677 and Education and Research Promotion program of KOREATECH in 2021 (H. J. Choi).
This work was supported by the NRF (Republic of Korea) grants No. 2021R1F1A1059671 (W. Choi), and No. 2022R1F1A1061968 (Y. Koh).

%%%%%%%%%%%%%%%%%%%%%%%%%%%%%%%%%%%%%%%%%%%%%%%%%%%%%%%%%%%%%%%%%%%%%%%%%%%%%%%%%%%%%%%%%%%%%%%%%%%%%

\end{document}